\documentclass[oneside,american,english]{amsart}
\usepackage[T1]{fontenc}
\usepackage[latin9]{inputenc}
\usepackage{geometry}
\usepackage{textcomp}
\usepackage{mathrsfs}
\usepackage{amstext}
\usepackage{amsthm}
\usepackage{amssymb}

\makeatletter
\numberwithin{equation}{section}
\numberwithin{figure}{section}
\theoremstyle{plain}
\newtheorem{thm}{\protect\theoremname}[section]
\theoremstyle{plain}
\newtheorem{conjecture}[thm]{\protect\conjecturename}
\theoremstyle{plain}
\newtheorem{cor}[thm]{\protect\corollaryname}
\theoremstyle{remark}
\newtheorem{rem}[thm]{\protect\remarkname}
\theoremstyle{plain}
\newtheorem{lem}[thm]{\protect\lemmaname}
\theoremstyle{definition}
\newtheorem{example}[thm]{\protect\examplename}
\theoremstyle{plain}
\newtheorem{prop}[thm]{\protect\propositionname}

\def\makebbb#1{
    \expandafter\gdef\csname#1\endcsname{
        \ensuremath{\Bbb{#1}}}
}\makebbb{R}\makebbb{N}\makebbb{Z}\makebbb{C}\makebbb{H}\makebbb{E}\makebbb{H}\makebbb{P}\makebbb{B}\makebbb{Q}\makebbb{E}\makebbb{E}
\makebbb{H}\makebbb{G}\makebbb{F}
\usepackage{babel}

\usepackage{babel}

\makeatother

\usepackage{babel}

\makeatother

\usepackage{babel}
\addto\captionsamerican{\renewcommand{\conjecturename}{Conjecture}}
\addto\captionsamerican{\renewcommand{\corollaryname}{Corollary}}
\addto\captionsamerican{\renewcommand{\examplename}{Example}}
\addto\captionsamerican{\renewcommand{\lemmaname}{Lemma}}
\addto\captionsamerican{\renewcommand{\propositionname}{Proposition}}
\addto\captionsamerican{\renewcommand{\remarkname}{Remark}}
\addto\captionsamerican{\renewcommand{\theoremname}{Theorem}}
\addto\captionsenglish{\renewcommand{\conjecturename}{Conjecture}}
\addto\captionsenglish{\renewcommand{\corollaryname}{Corollary}}
\addto\captionsenglish{\renewcommand{\examplename}{Example}}
\addto\captionsenglish{\renewcommand{\lemmaname}{Lemma}}
\addto\captionsenglish{\renewcommand{\propositionname}{Proposition}}
\addto\captionsenglish{\renewcommand{\remarkname}{Remark}}
\addto\captionsenglish{\renewcommand{\theoremname}{Theorem}}
\providecommand{\conjecturename}{Conjecture}
\providecommand{\corollaryname}{Corollary}
\providecommand{\examplename}{Example}
\providecommand{\lemmaname}{Lemma}
\providecommand{\propositionname}{Proposition}
\providecommand{\remarkname}{Remark}
\providecommand{\theoremname}{Theorem}

\begin{document}
\title{Sharp bounds on the height of K-semistable Fano varieties I, the toric
case}
\begin{abstract}
Inspired by K. Fujita's algebro-geometric result that complex projective
space has maximal degree among all K-semistable complex Fano varieties,
we conjecture that the height of a K-semistable metrized arithmetic
Fano variety $\mathcal{X}$ of relative dimension $n$ is maximal
when $\mathcal{X}$ is the projective space over the integers, endowed
with the Fubini-Study metric. Our main result establishes the conjecture
for the canonical integral model of a toric Fano variety when $n\leq6$
(the extension to higher dimensions is conditioned on a conjectural
``gap hypothesis'' for the degree). Translated into toric Kähler
geometry this result yields a sharp lower bound on a toric invariant
introduced by Donaldson, defined as the minimum of the toric Mabuchi
functional. We furthermore reformulate our conjecture as an optimal
lower bound on Odaka's modular height. In any dimension $n$ it is
shown how to control the height of the canonical toric model $\mathcal{X},$
with respect to the Kähler-Einstein metric, by the degree of $\mathcal{X}.$
In a sequel to this paper our height conjecture is established for
any projective diagonal Fano hypersurface, by exploiting a more general
logarithmic setup.\thanks{2020 Mathematics Subject Classification 14G40 (primary), 32Q20, 53C25,
11G50, 14J45 (secondary)}
\end{abstract}

\author{Rolf Andreasson, Robert J. Berman}
\email{rolfan@chalmers.se, robertb@chalmers.se}
\curraddr{Chalmers tvärgata 3, 412 96 Göteborg, Chalmers tvärgata 3, 412 96
Göteborg}
\maketitle

\section{Introduction}

\subsection{\label{subsec:The-height-of}The height of K-semistable Fano varieties}

Let $(\mathcal{X},\mathcal{L})$ be a projective flat scheme $\mathcal{X}$
over $\Z$ of relative dimension $n,$ endowed with a relatively ample
line bundle $\mathcal{L}.$ The complexification of $(\mathcal{X},\mathcal{L})$
will be denoted by $(X,L).$ In other other words, $X$ is the complex
projective variety consisting of the complex points of $\mathcal{X}$
and $L$ is the corresponding ample line bundle over $X.$ 

A central role in arithmetic and Diophantine geometry is played by
the\emph{ height} of $(\mathcal{X},\mathcal{L}),$ which is defined
with respect to a continuous metric $\left\Vert \cdot\right\Vert $
on $L.$ This is an arithmetic analog of the algebro-geometric degree
of $(X,L),$ i.e., of the top intersection number $L^{n}$ on $X.$
The height of $(\mathcal{X},\mathcal{L},\left\Vert \cdot\right\Vert )$
- also known as Faltings' height - is defined as the $(n+1)-$fold
arithmetic intersection number of the metrized line bundle $(\mathcal{L},\left\Vert \cdot\right\Vert )$
on $\mathcal{X},$ introduced by Gillet-Soulé in the context of Arakelov
geometry \cite{fa,b-g-s} (see Section \ref{subsec:The-height-of}).
We recall that in Arakelov geometry the metric $\left\Vert \cdot\right\Vert $
on $L$ plays the role of a ``compactification'' of $\mathcal{X}.$
Accordingly, a metrized line bundle $(\mathcal{L},\left\Vert \cdot\right\Vert )$
is usually denoted by $\overline{\mathcal{L}}.$ The definition of
height naturally extends to any $\Q-$line bundle $\mathcal{L},$
using homogeneity. 

In contrast to the algebro-geometric degree of $L$ the height of
$\overline{\mathcal{L}}$ can rarely be computed explicitly and all
one can hope for, in general, is explicit bounds on the height. When
$\mathcal{L}$ is the relative canonical line bundle, that we shall
denote by $\mathcal{K}_{\mathcal{X}}$ and $n=1,$ such conjectural
upper bounds are motivated by the Bogolomov-Miyaoka-Yau inequality
on $X$ and imply, in particular, the effective Mordell conjecture,
concerning explicit upper bounds on the number of rational points
on $X_{\Q}$ and the abc-conjecture \cite{pa,v,sou0}. Here we shall
be concerned with the opposite situation where $\mathcal{X}$ is an\emph{
arithmetic Fano variety}, in the sense that the relative anti-canonical
line bundle is defined as a relative ample $\Q-$line bundle that
we denote by $\mathcal{-K}_{\mathcal{X}},$ using additive notation
for tensor products (see Section \ref{subsec:Arithmetic-Fano-varieties}).
In particular, $X$ is a complex\emph{ Fano variety}; a variety whose
canonical line bundle $-K_{X}$ defines an ample $\Q-$line bundle.
We will also, for simplicity, assume that $X$ is normal. As shown
in \cite{ber-ber2} in the toric case and then \cite{fu} in general,
for any complex Fano variety $X$

\begin{equation}
(-K_{X})^{n}\leq(-K_{\P_{\C}^{n}})^{n}\label{eq:fujita intro}
\end{equation}
under the assumption that $X$ is \emph{K-semistable.} Moreover, equality
holds iff $X=\P_{\C}^{n}$ \cite{liu}. In contrast, when $X$ is
not K-semistable the degree $(-K_{X})^{n}$ gets arbitrarily large
in any given dimension $n,$ for singular $X$ (see \cite[Ex 4.2]{de}
for simple two-dimensional toric examples). The notion of K-stability
first arose in the context of the Yau-Tian-Donaldson conjecture for
Fano manifolds, saying that a Fano manifold admits a Kähler-Einstein
metric if and only if it is K-polystable \cite{ti,do1}. The conjecture
was settled in \cite{c-d-s} and very recently also established for
singular Fano varieties \cite{li1,l-x-z}. From a purely algebro-geometric
perspective K-stability can be viewed as a limiting form of Chow and
Hilbert-Mumford stability \cite{r-t}, that enables a good theory
of moduli spaces (see the survey \cite{x}). 

Is there an arithmetic analog of the inequality \ref{eq:fujita intro}?
More precisely, it seems natural to ask if, under appropriate assumptions,
the height $(\overline{-\mathcal{K}_{\mathcal{X}}})^{n+1}$ is bounded
from above by the height $(\overline{-\mathcal{K}_{\P_{\Z}^{n}}})^{n+1}$
of the relative anti-canonical line bundle on the projective space
$\P_{\Z}^{n}$ over the integers, endowed with its standard Kähler-Einstein
metric (the Fubini-Study metric)? This would yield an explicit bound
on the height $(\overline{-\mathcal{K}_{\mathcal{X}}})^{n+1}$$,$
since the height of Fubini-Study metric on projective space was explicitly
calculated in \cite[ §5.4]{g-s}, giving, after volume-normalization,
\begin{equation}
(\overline{-\mathcal{K}_{\P_{\Z}^{n}}})^{n+1}=\frac{1}{2}(n+1)^{n+1}\left((n+1)\sum_{k=1}^{n}k^{-1}-n+\log(\frac{\pi^{n}}{n!})\right)\label{eq:expl formul on p n}
\end{equation}
If such a universal bound is to hold one needs, however, to impose
a normalization condition on the metric on $-K_{X}.$ Indeed, $\overline{\mathcal{L}}^{n+1}$
is additively equivariant with respect to scalings of the metric.
Accordingly, the metric $\left\Vert \cdot\right\Vert $ on $-K_{X}$
will henceforth be assumed to be \emph{volume-normalized }in the sense
that the corresponding volume form on $X$ has total unit volume.
As it turns out, the supremum of the height $\overline{-\mathcal{K}_{\mathcal{X}}}^{n+1}$
over all volume-normalized metrics on $-K_{X}$ with positive curvature
current is finite if and only if $X$ is K-semistable (Theorem \ref{thm:arithm Vol and K semi st}).
It seems thus natural to make the following conjecture:
\begin{conjecture}
\label{conj:height intro}Let $\mathcal{X}$ be an arithmetic Fano
variety of relative dimension $n$ over $\Z.$ If the complexification
$X$ of $\mathcal{X}$ is K-semistable, then the following height
inequality holds for any volume-normalized continuous metric on $-K_{X}$
with positive curvature current:
\[
(\overline{-\mathcal{K}_{\mathcal{X}}})^{n+1}\leq(\overline{-\mathcal{K}_{\P_{\Z}^{n}}})^{n+1},
\]
where $-K_{\P_{\C}^{n}}$ is endowed with the volume normalized Fubini-Study
metric. Moreover, if $\mathcal{X}$ is normal equality holds if and
only if $\mathcal{X}=\P_{\Z}^{n}$ and the metric is Kähler-Einstein,
i.e. coincides with the Fubini-Study metric, modulo the action of
an automorphism.
\end{conjecture}

More generally, when $\Z$ is replaced by the ring of integers of
a number field $F,$ i.e. a finite field extension $F$ of $\Q,$
the height $(\overline{-\mathcal{K}_{\mathcal{X}}})^{n+1}$ should
be divided by the  degree $[F:\Q].$ But, for simplicity, we will
focus on the case when $F=\Q$ (see Section \ref{subsec:Outlook-on-a}
for a generalization of the previous conjecture). The converse ``only
if'' statement to the previous conjecture does hold (as a consequence
of Theorem \ref{thm:arithm Vol and K semi st}). Moreover, the conjecture
is compatible with taking products (Prop \ref{prop:product}). The
inequality in the previous conjecture is equivalent to the following
inequality for any continuous metric on $-K_{X}$ with positive curvature
current, as follows from a simple scaling argument, 
\begin{equation}
\frac{(\overline{-\mathcal{K}_{\mathcal{X}}})^{n+1}}{(n+1)}+\frac{(-K_{X})^{n}}{2}\log\mu(X)\leq c_{n}\label{eq:conj intro with volume}
\end{equation}
 where $\mu(X)$ denotes the volume of $X$ with respect to the measure
$\mu$ on $X$ corresponding to the metric $\left\Vert \cdot\right\Vert $
on $-K_{X}$ and $c_{n}$ denotes the constant in the right hand side
of formula \ref{eq:expl formul on p n}. Some intruiging relations
between the conjectural bound \ref{eq:conj intro with volume} and
the Manin-Peyre conjecture, concerning the density of rational points
on Fano varieties, are discussed in \cite{ber1}.

Our main result concerns the case when $X$ is toric and $\mathcal{X}$
is its canonical toric integral model (see \cite[Section 2]{ma} and
\cite[Def 3.5.6]{b-g-p-s}).
\begin{thm}
\label{thm:main toric intro}Let $X$ be an $n-$dimensional K-semistable
toric Fano variety and denote by $\mathcal{X}$ its canonical model
over $\Z.$ Then the previous conjecture holds under anyone of the
following conditions:
\begin{itemize}
\item $n\leq6$ and $X$ is $\Q-$factorial (equivalently, $X$ is non-singular
or has abelian quotient singularities) 
\item $X$ is not Gorenstein or has some abelian quotient singularity
\end{itemize}
\end{thm}

Note that when $n=2$ any toric variety is, in fact, $\Q-$factorial.
More generally, we will show that the curvature assumption may be
dispensed with if the height $(\overline{-\mathcal{K}_{\mathcal{X}}})^{n+1}$
is replaced by the $\chi-$arithmetic volume $\widehat{\text{vol}}_{\chi}\left(\overline{-\mathcal{K}_{\mathcal{X}}}\right)$
of $\overline{-\mathcal{K}_{\mathcal{X}}}$ (whose definition is recalled
in Section \ref{subsec:The-arithmetic-volume}). We expect that the
maximum of $\widehat{\text{vol}}_{\chi}\left(\overline{-\mathcal{K}_{\mathcal{X}}}\right)$
over all integral models $(\mathcal{X},-\mathcal{K}_{\mathcal{X}})$
of a given toric Fano variety $(X,-K_{X})$ is attained at the canonical
integral model $\mathcal{X}$ featuring in the previous theorem. This
expectation is inspired by a conjecture of Odaka discussed in Section
\ref{subsec:The-arithmetic-K-energy} below. 

The key ingredient in the proof of Theorem \ref{thm:main toric intro}
is the following bound estimating the arithmetic volume $\widehat{\text{vol}}_{\chi}\left(\overline{-\mathcal{K}_{\mathcal{X}}}\right)$
of any volume-normalized metric on $-K_{X}$ in terms of the algebro-geometric
volume $\text{vol\ensuremath{(X)} }$(Prop \ref{prop:universal toric bound}):
\begin{equation}
\widehat{\text{vol}}_{\chi}\left(\overline{-\mathcal{K}_{\mathcal{X}}}\right)\leq-\frac{1}{2}\text{vol}(X)\log\left(\frac{\text{vol}(X)}{(2\pi^{2})^{n}}\right)\,\,\text{vol}(X):=(-K_{X})^{n}/n!\label{eq:universal bound intro}
\end{equation}
Since $\text{vol}(X)$ is maximal for $X=\P^{n}$ the right hand side
above is bounded by a constant $C_{n}$ only depending on the dimension
$n.$ Under the ``gap hypothesis'' that $\P^{n-1}\times\P^{1}$
has the second largest volume among all $n-$dimensional K-semistable
$X$ we show that the bound \ref{eq:universal bound intro} implies
Conjecture \ref{conj:height intro} for the canonical integral model
$\mathcal{X}$ of a toric Fano variety $X.$ The proof of Theorem
\ref{thm:main toric intro} is concluded by verifying the gap hypothesis
under the conditions in Theorem \ref{thm:main toric intro}. But we
do expect that the gap hypothesis above holds for any toric Fano variety
(see Section \ref{subsec:gap}).

In a sequel \cite{a-b} to the present paper Conjecture \ref{conj:height intro}
is established for any diagonal Fano hypersurface $\mathcal{X}$ in
$\P_{\Z}^{n+1}$ (i.e. $\mathcal{X}$ is the subscheme cut out by
a homogeneous polynomial of the form $a_{0}x_{0}^{d}+...+a_{n+1}x_{n+1}^{d}$
for any given integers $a_{i},$ with no common divisors, and $d\leq n+1).$
Although $\mathcal{X}$ is not toric the proof, somewhat surprisingly,
is reduced to a simple toric logarithmic case.

\subsection{The height of toric Kähler-Einstein metrics }

In the toric case, $X$ is K-semistable if and only if it is K-polystable
and thus admits a toric Kähler-Einstein metric \cite{w-z,ber-ber},
i.e. a toric continuous metric on $-K_{X}$ whose curvature form defines
a Kähler metric with constant positive Ricci curvature on the regular
locus of $X.$ Moreover, in general, any volume-normalized Kähler-Einstein
metric maximizes $(\overline{-\mathcal{K}_{\mathcal{X}}})^{n+1}.$
This means that the inequality in the previous theorem is equivalent
to the corresponding inequality for the volume-normalized toric Kähler-Einstein
metric on $-K_{X}.$ The special role of the Kähler-Einstein condition
in arithmetic (Arakelov) geometry - as an analog of minimality of
$\mathcal{X}$ over $\text{Spec \ensuremath{\Z} - }$ was emphasized
already in  the early days of Arakelov geometry by Manin \cite{man}.
It is, however, rare that the Kähler-Einstein metric and the corresponding
height, can be explicitly computed. In fact, in the Fano case this
seems to only have been achieved when $X$ is homogeneous \cite{ma2,c-m,k-k,ta1,ta2,ta3}.
The following result, complementing the general upper bound \ref{eq:universal bound intro},
yields a rather precise control on its height $(\overline{-\mathcal{K}_{\mathcal{X}}})^{n+1}$
in the toric case:
\begin{thm}
\label{thm:ke intro}Let $X$ be an $n-$dimensional toric Fano variety
and denote by $\mathcal{X}$ its canonical model over $\Z.$ Then
the height $(\overline{-\mathcal{K}_{\mathcal{X}}})_{\text{}}^{n+1}$
of any volume-normalized Kähler-Einstein metric satisfies
\[
\frac{(n+1)!}{2}\mathrm{vol}(X)\log\left(\frac{n!m_{n}\pi^{n}}{\mathrm{vol}(X)}\right)\leq(\overline{-\mathcal{K}_{\mathcal{X}}})^{n+1}\leq\frac{(n+1)!}{2}\mathrm{vol}(X)\log\left(\frac{(2\pi)^{n}\pi^{n}}{\mathrm{vol}(X)}\right)
\]
where $m_{n}$ denotes the largest lower bound on the Mahler volume
of a convex body. In particular, $(\overline{-\mathcal{K}_{\mathcal{X}}})^{n+1}>0.$ 
\end{thm}

We also provide an infinite family of toric varities $X$ for which
the height of the corresponding Kähler-Einstein can be explicitely
computed as a function $f(v)$ of $\mathrm{vol}(X)$ of the same form
as in the previous theorem; $f(v)=v\log(av^{-1})$ for some constant
$a.$ The constant $m_{n}$ in the previous theorem is the largest
constant satisfying 
\[
m_{n}\leq\text{vol}(P)\text{vol}(P^{*}),
\]
 where $P^{*}$ denotes the polar dual of any given convex body $P$
containing the origin in its interior (the role of $P$ in the present
setting is played by the moment polytope of $X).$ According to Mahler's
conjecture, the constant $m_{n}$ is equal to $(n+1)^{n+1}/(n!)^{2}$
(which is realized for a simplex $P$). The case $n=2$ was settled
in \cite{mah}, but for our purposes the following general bound from
\cite{ku} will be enough: 
\[
m_{n}\geq(\frac{\pi}{2e})^{n-1}(n+1)^{n+1}/(n!)^{2},
\]
which implies the strict positivity of $(\overline{-\mathcal{K}_{\mathcal{X}}})^{n+1}.$
Combining the previous theorem with the upper bound \ref{eq:fujita intro}
thus yields the following universal bounds:
\begin{cor}
\label{cor:universal ke}Let $X$ be an $n-$dimensional toric Fano
variety and denote by $\mathcal{X}$ its canonical model over $\Z.$
Then the height $(\overline{-\mathcal{K}_{\mathcal{X}}})_{\text{}}^{n+1}$
of any volume-normalized Kähler-Einstein metric satisfies the following
universal bounds
\[
0<(\overline{-\mathcal{K}_{\mathcal{X}}})^{n+1}\leq\frac{n(n+1)^{n+1}}{2}\log\left(\frac{2\pi^{2}n!}{n+1}\right)
\]
\end{cor}

Incidentally, the upper bound above is related to a question posed
in \cite{na}, asking whether $(\overline{-\mathcal{K}_{\mathcal{X}}})^{n+1}$
is bounded from above by a universal constant $C_{n},$ under the
assumption that $X$ be non-singular and $\overline{-\mathcal{K}_{\mathcal{X}}}$
be relatively ample. This is a stronger condition than having positive
curvature, as we assume. We also allow singularities, but our results
concern only the toric case. Under the conditions in Theorem \ref{thm:main toric intro}
our upper bound may be improved to the sharp bound $(\overline{-\mathcal{K}_{\P_{\Z}^{n}}})^{n+1}$
(given by formula \ref{eq:expl formul on p n}). As for the lower
bound it is sharp in any dimension $n$. Indeed, there are $n-$dimensional
K-semistable ($\Q$-factorial) Fano varieties $X$ such that $\text{vol}(X)$
and thus (by Theorem \ref{thm:ke intro}) $(\overline{-\mathcal{K}_{\mathcal{X}}})^{n+1}$
is arbitrarily close to $0;$ see Example \ref{exa:toric family}. 

\subsection{\label{subsec:Donaldson's-toric-invariant}Donaldson's toric invariant}

Let now $(X,L)$ be a polarized complex projective manifold. A prominent
role in Kähler geometry is played by Mabuchi's K-energy functional
$\mathcal{M}$ \cite{mab}, defined on the space $\mathcal{H}(X,L)$
of all smooth metrics $\left\Vert \cdot\right\Vert $ on $L$ with
positive curvature. Its critical points are the metrics whose curvature
form $\omega$ define a Kähler metric on $X$ with constant scalar
curvature. The precise definition of $\mathcal{M}$ is recalled in
Section \ref{subsec:The-Mabuchi-functional}. Since the definition
of $\mathcal{M}$ only involves its differential, the functional $\mathcal{M}$
is only defined up to addition by a real constant. However, when $(X,L)$
is toric Donaldson \cite{do1} exploited the toric structure to define
the Mabuchi functional $\mathcal{M}$ as a canonical functional on
toric metrics: 
\begin{equation}
\mathcal{M}_{L}:=\int_{\partial P}ud\sigma-a\int_{P}udx-\int_{P}\log\det(\nabla^{2}u)dx,\,\,\,\,\,\,a:=\int_{\partial P}d\sigma/\int_{P}dx\label{eq:Ds toric Mab intro}
\end{equation}
 where $P$ is the moment polytope in $\R^{n}$ corresponding to the
polarized toric manifold $(X,L),$ whose boundary $\partial P$ comes
with a measure $d\sigma$ induced by Lebesgue measure $dx$ on $\R^{n}$
and the lattice $\Z^{n}$ in $\R^{n}$ and $u$ is the smooth bounded
convex function on $P$ corresponding to a toric metric on $L$ under
Legendre transformation (see Section \ref{subsec:Logarithmic-coordinates-and}).
In particular, in the last section of \cite{do1} Donaldson introduced
an invariant of a polarized toric manifold $(X,L),$ defined as the
infimum of the toric Mabuchi functional $\mathcal{M}_{L}$ defined
by formula \ref{eq:Ds toric Mab intro}. Here we show that Theorem
\ref{thm:main toric intro} implies that when $X$ is a Fano variety
and $L=-K_{X}$ a slight perturbation of Donaldson's invariant is
minimal when $X$ is complex projective space, under the conditions
on $X$ appearing in Theorem \ref{thm:main toric intro}:
\begin{thm}
\label{thm:Don inv intro}Let $X$ be a K-semistable toric Fano variety
of dimension $n,$ satisfying the conditions in Theorem \ref{thm:main toric intro}.
Then the invariant 
\[
\text{\ensuremath{X\mapsto\inf_{\mathcal{H}(X,-K_{X})}\mathcal{M}_{-K_{X}}-\frac{(-K_{X})^{n}}{n!}\log\left(\frac{(-K_{X})^{n}}{n!}\right)} }
\]
is minimal for $X=\P^{n}$ (and only then), where the inf is attained
at the metric on $-K_{\P^{n}}$ induced by the Fubini-Study metric.
\end{thm}

In the previous theorem the Fano variety $X$ is allowed to be singular.
The Mabuchi functional for singular general Fano varieties was introduced
in \cite{d-t,bbegz} and Donaldson's formula \ref{eq:Ds toric Mab intro}
was extended to singular toric Fano varieties in \cite{ber-ber}.
In general, for Fano varieties the Mabuchi functional $\mathcal{M}$
is bounded from below iff $X$ is K-semistable \cite{li} (see the
discussion following Theorem \ref{thm:arithm Vol and K semi st}). 

\subsection{\label{subsec:The-arithmetic-K-energy}The arithmetic Mabuchi functional
and Odaka's modular height}

For a general polarized manifold $(X,L)$ the infimum of the Mabuchi
functional $\mathcal{M}$ is not canonically defined (since $\mathcal{M}$
is only defined up to addition by a constant). But to any given integral
model $(\mathcal{X},\mathcal{L})$ of a polarized complex variety
$(X,L)$ one may, as shown by Odaka \cite{o}, attach a particular
Mabuchi functional $\mathcal{M}_{(\mathcal{X},\mathcal{L})}$ which
(up to a multiplicative normalization) is given as the following sum
of arithmetic intersection numbers:
\begin{equation}
\mathcal{M}_{(\mathcal{X},\mathcal{L})}(\overline{\mathcal{L}}):=\frac{a}{(n+1)!}\overline{\mathcal{L}}^{n+1}-\frac{1}{n!}(-\overline{\mathcal{K}}_{\mathcal{X}})\cdot\overline{\mathcal{L}}^{n},\,\,\,\,a=-n(K_{X}\cdot L^{n-1})/L^{n}\label{eq:def of arithm Mab intro}
\end{equation}
 where, as in the previous section, $\overline{\mathcal{L}}$ denotes
the metrized line bundle $(\mathcal{L},\left\Vert \cdot\right\Vert ).$
In the definition of the second arithmetic intersection number above
one also needs to endow $-K_{X}$ with a metric and one is confronted
with two different natural choices: either the metric induced by the
volume form $\omega^{n}/n!$ of the Kähler metric $\omega$ defined
by the curvature form of $(\mathcal{L},\left\Vert \cdot\right\Vert )$
or the\emph{ normalized} volume form $\omega^{n}/L^{n}$ (which has
unit total volume). The first choice is the one adopted in \cite{o}
and we show that when $X$ is a toric Fano variety and $(\mathcal{X},\mathcal{L})$
is the canonical integral model of ($X,L)$ this choice coincides
with Donaldson's one (formula \ref{eq:Ds toric Mab intro}). However,
for our purposes the second volume-normalized choice turns out to
be the appropriate one. It yields, in particular, the shift by the
logarithm of $(-K_{X})^{n}$ appearing in Theorem \ref{thm:Don inv intro}:
\[
2\mathcal{M}_{(\mathcal{X},\mathcal{-}\mathcal{K}_{X})}=\mathcal{M}_{-K_{X}}-\frac{(-K_{X})^{n}}{n!}\log\left(\frac{(-K_{X})^{n}}{n!}\right)
\]
(Prop \ref{prop:arithm Mab as Don Mab}). The point is that with this
choice the following formula holds in the arithmetic setting:
\begin{equation}
\sup\frac{(\overline{-\mathcal{K}_{\mathcal{X}}})^{n+1}}{(n+1)!}=-\inf_{\mathcal{H}(X,-K_{X})}\mathcal{M}_{(\mathcal{X},-\mathcal{K}_{\mathcal{X}})}\label{eq:sup inf intro}
\end{equation}
 where the sup ranges over all volume-normalized metrics in $\mathcal{H}(X,-K_{X})$
(see Prop \ref{prop:inf arithm Mab vs Ding}). As a consequence, Conjecture
\ref{conj:height intro} is equivalent to the inequality
\begin{equation}
\inf_{\mathcal{H}(X,-K_{X})}\mathcal{M}_{(\mathcal{X},-K_{\mathcal{X}})}\geq\inf_{\mathcal{H}(\P^{n},-K_{\P^{n}})}\mathcal{M}_{(\P_{\Z}^{n},...)}.\label{eq:lower bound arith Mab intro}
\end{equation}
Theorem \ref{thm:Don inv intro} thus follows from Theorem \ref{thm:main toric intro}. 

\subsubsection{\label{subsec:Odaka's-modular-height}Odaka's modular height}

Let $(X_{F},L_{F})$ be an $n-$dimensional polarized variety defined
over a number field $F.$ In \cite{o} Odaka introduced the following
invariant of $(X_{F},L_{F}),$ dubbed the \emph{intrinsic K-modular
height of $(X_{F},L_{F}):$} 

\begin{equation}
h(X_{F},L_{F})=\inf_{(\mathcal{X},\mathcal{L})}\inf_{\mathcal{H}(X,L)}\mathcal{M}_{(\mathcal{X},\mathcal{L})},\label{eq:def of Odakas mod inv intr}
\end{equation}
 where $(\mathcal{X},\mathcal{L})$ is a model of $(X_{F},L_{F})$
over the rings of integers $\mathcal{O}_{F'}$ of a finite field extension
$F'$ of $F$ and $\mathcal{M}_{(\mathcal{X},\mathcal{L})}$ now denotes
the arithmetic K-energy \ref{eq:def of arithm Mab intro}, divided
by the degree $[F':\Q].$ In contrast to \cite{o}, we will employ
the volume-normalized metric on $-K_{X}$ in the definition of $\mathcal{M}_{(\mathcal{X},\mathcal{L})},$
discussed in the previous section. As shown in \ref{eq:def of arithm Mab intro},
for a polarized abelian variety $(X_{F},L_{F})$, Odaka's modular
height $h(X_{F},L_{F})$ essentially coincides with Faltings' stable
modular height of $(X_{K},L_{K})$ \cite{fa00} (see Section \ref{subsec:Comparison-with-Odaka's Falting}).
Furthermore, as explained in \cite{o}, $h(X_{F},L_{F})$ can be viewed
as a ``large rank limit'' of Bost's and Zhang's intrinsic heights
appearing in \cite{bo1,bo2,zh1}, where the role of K-semistability
is played by Chow semistability (see formula \ref{eq:Od asym}). We
propose the following
\begin{conjecture}
\label{conj:min of Odaka for Fano}Let $X_{\Q}$ be a Fano variety
defined over $\Q.$ Then Odaka's modular invariant $h(X_{\Q},-K_{X_{\Q}}),$
normalized as above, is minimal when $X_{\Q}=\P_{\Q}^{n}.$
\end{conjecture}

According to a conjecture of Odaka \cite{od2} any globally K-semistable
integral model $(\mathcal{X},-\mathcal{K}_{\mathcal{X}})$ of $(X,-K_{X})$
minimizes $\mathcal{M}_{(\mathcal{X},\mathcal{L})}$ over all models
$(\mathcal{X},\mathcal{L})$ (the function field analog of this minimization
property is established in \cite{b-x}; see also \cite[Remark 7.9]{x}).
Global K-semistability means that all the fibers of $\mathcal{X}\rightarrow\text{Spec \ensuremath{\mathcal{O}_{F}}}$
are K-semistable. In other words, in addition to the K-semistability
of the generic fiber $X_{F}$ this means that the variety $X_{\mathbb{F}_{p}}$
over the finite field $\mathbb{F}_{p},$ corresponding to the integral
model $\mathcal{X},$ is K-semistable for any prime ideal $p.$ For
example, as pointed out to us by Odaka the canonical model $\mathcal{X}$
of a K-semistable toric Fano variety $X_{\Q}$ appearing in Theorem
\ref{thm:main toric intro} is globally K-semistable. Thus if Odaka's
minimization conjecture holds, then Theorem \ref{thm:main toric intro}
implies Conjecture \ref{conj:min of Odaka for Fano} for any toric
Fano variety $X_{\Q}$ satisfying the conditions in Theorem \ref{thm:main toric intro}.\footnote{During the revision of the first preprint version of the present paper
Odaka's minimization conjecture was settled in \cite{h-o} under slightly
stronger assumptions than global K-semistability. } Anyhow, the positivity statement in Theorem \ref{thm:ke intro} implies
that the modular invariant $h(X_{\Q},-K_{X_{\Q}})$ is negative for
any K-semistable toric Fano variety $X_{\Q}.$ 

\subsection{Organization}

In Section \ref{sec:Heights,-arithmetic-volumes} we start by recalling
the complex-geometric and arithmetic setup before proving Theorem
\ref{thm:arithm Vol and K semi st}, relating upper bounds on the
height of Fano varieties to K-semistability. The proof leverages an
arithmetic analog of the Ding functional. In Section \ref{sec:Sharp-height-inequalities}
we specialize to the toric situation and prove the sharp height inequality
in Theorem \ref{thm:main toric intro}, stated in the introduction
and the height bounds for Kähler-Einstein metrics in Theorem \ref{thm:ke intro}.
We also show that Conjecture \ref{conj:height intro} is compatible
with taking products. We then go on, in Section \ref{sec:Donaldson's-toric-invariant},
to deduce Theorem \ref{thm:Don inv intro} concerning the sharp lower
bound on Donaldson's toric Mabuchi functional. In Section \ref{sec:Connections-to-the ar}
Donaldson's functional is related to Odaka's arithmetic Mabuchi functional,
which, in turn is related to the arithmetic Ding functional. In the
last section we make a comparison with the function field case, formulate
a generalized version of Conjecture \ref{conj:height intro} and compare
with previous work of Bost and Zhang, Odaka and Faltings.

We have made an effort to make the paper readable for the reader with
a background in arithmetic geometry, as well as for the complex geometers,
by including most of the background material needed for the proofs
of the main results.

\subsection{Acknowledgements}

We are grateful to Bo Berndtsson, Dennis Eriksson, Gerard Freixas
i Montplet, Benjamin Nill, Yuji Odaka, Per Salberger, Chenyang Xu
and Ziquan Zhuang for illuminating discussions/comments and, in particular,
to Alexander Kasprzyk for updating the database \cite{ob}. We are
also grateful to the refere for very helpful comments. This work was
supported by grants from the Knut and Alice Wallenberg foundation,
the Göran Gustafsson foundation and the Swedish Research Council.

\section{\label{sec:Heights,-arithmetic-volumes}Heights, arithmetic volumes
and K-stability of Fano varieties}

In this section we show, in particular, that the height of a polarized
integral model $(\mathcal{X},\mathcal{L})$ of a Fano manifold $(X,-K_{X})$
is bounded from above - as the metric on $\mathcal{L}$ ranges over
all volume-normalized metrics with positive curvature current - if
and only if $(X,-K_{X})$ is K-semistable (Theorem \ref{thm:arithm Vol and K semi st}).
See also \cite{o} for further connections between K-stability of
polarized varieties $(X,L)$ and arithmetic geometry. The main new
feature here, compared to \cite{o}, is that we leverage an arithmetic
version of the Ding functional in Kähler geometry, while \cite{o}
considers an arithmetic version of the Mabuchi functional (the two
functionals are compared in Section \ref{sec:Connections-to-the ar}). 

\subsection{Complex geometric setup}

Throughout the paper $X$ will denote a compact connected complex
normal variety, assumed to be $\Q-$Gorenstein. This means that the
canonical divisor $K_{X}$ on $X$ is defined as a $\Q-$line bundle:
there exists some positive integer $m$ and a line bundle on $X$
whose restriction to the regular locus $X_{\text{reg }}$ of $X$
coincides with the $m$:th tensor power of $K_{X_{\text{reg }}},$
i.e. the top exterior power of the cotangent bundle of $X_{\text{reg }}.$
We will use additive notation for tensor powers of line bundles.

\subsubsection{\label{subsec:Metrics-on-line}Metrics on line bundles}

Let $(X,L)$ be a polarized complex projective variety i.e. a complex
normal variety $X$ endowed with an ample line bundle $L.$ We will
use additive notation for metrics on $L.$ This means that we identify
a continuous Hermitian metric $\left\Vert \cdot\right\Vert $ on $L$
with a collection of continuous local functions $\phi_{U}$ associated
to a given covering of $X$ by open subsets $U$ and trivializing
holomorphic sections $e_{U}$ of $L\rightarrow U:$ 
\begin{equation}
\phi_{U}:=-\log(\left\Vert e_{U}\right\Vert ^{2}),\label{eq:def of phi U}
\end{equation}
 which defines a function on $U.$ Of course, the functions $\phi_{U}$
on $U$ do not glue to define a global function on $X,$ but the current
\[
dd^{c}\phi_{U}:=\frac{i}{2\pi}\partial\bar{\partial}\phi_{U}
\]
 is globally well-defined and coincides with the normalized curvature
current of $\left\Vert \cdot\right\Vert $ (the normalization ensures
that the corresponding cohomology class represents the first Chern
class $c_{1}(L)$ of $L$ in the integral lattice of $H^{2}(X,\R)).$
Accordingly, as is customary, we will symbolically denote by $\phi$
a given continuous Hermitian metric on $L$ and by $dd^{c}\phi$ its
curvature current. The space of all continuous metrics $\phi$ on
$L$ will be denoted by $\mathcal{C}^{0}(L).$ We will denote by $\mathcal{C}^{0}(L)\cap\text{PSH\ensuremath{(L)}}$
the space of all continuous metrics on $L$ whose curvature current
is positive, $dd^{c}\phi\geq0$ (which means that $\phi_{U}$ is plurisubharmonic,
or psh, for short). Then the exterior powers of $dd^{c}\phi$ are
defined using the local pluripotential theory of Bedford-Taylor \cite{b-b}.
The \emph{volume} of an ample line bundle $L$ may be defined by

\begin{equation}
\text{vol}(L):=\lim_{k\rightarrow\infty}k^{-n}\dim H^{0}(X,L^{\otimes k})=\frac{1}{n!}L^{n}=\frac{1}{n!}\int_{X}(dd^{c}\phi)^{n}\label{eq:HS formully algebraic}
\end{equation}
using in the second equality the Hilbert-Samuel theorem and where
$\phi$ denotes any element in $\mathcal{C}^{0}(L)\cap\text{PSH\ensuremath{(L)}}.$

More generally, metrics $\phi$ are defined for a $\Q-$line bundle
$L:$ if $mL$ is a bona fide line bundle, for $m\in\Z_{+},$ then
$m\phi$ is a bona fide metric on $mL.$ 
\begin{rem}
The normalization of $\phi_{U}$ used here coincides with the one
in \cite{ber0,ber-ber}, but it is twice the one employed in \cite{b-b}.
\end{rem}

\subsubsection{\label{subsec:Metrics-on- minus KX vs volume}Metrics on $-K_{X}$
vs volume forms on $X$}

First consider the case when $X$ is smooth. Then any smooth metric
$\left\Vert \cdot\right\Vert $ on $-K_{X}$ corresponds to a volume
form on $X,$ defined as follows. Given local holomorphic coordinates
$z$ on $U\subset X$ denote by $e_{U}$ the corresponding trivialization
of $-K_{X},$ i.e. $e_{U}=\partial/\partial z_{1}\wedge\cdots\wedge\partial/\partial z_{n}.$
The metric on $-K_{X}$ induces, as in the previous section, a function
$\phi_{U}$ on $U$ and the volume form in question is locally defined
by
\begin{equation}
e^{-\phi_{U}}(\frac{i}{2})^{n^{2}}dz\wedge d\bar{z},\,\,\,\,\,dz:=dz_{1}\wedge\cdots\wedge dz_{n},\label{eq:volume form e minus phi U}
\end{equation}
on $U,$ which glues to define a global volume form on $X.$ In other
words, $e^{-\phi_{U}}$ is the density of the volume form with respect
to the local Euclidean volume form. Accordingly, we will simply denote
the volume form in question by $e^{-\phi}$, abusing notation slightly.
When $X$ is singular any continuous metric $\phi$ on $-K_{X}$ induces
a measure on $X,$ symbolically denoted by $e^{-\phi},$ defined as
before on the regular locus $X_{\text{reg }}$ of $X$ and then extended
by zero to all of $X.$ We will say that a measure $dV$ on $X$ is
a continuous volume form if it corresponds to a continuous metric
on $-K_{X}.$ A Fano variety has log terminal singularities iff it
admits a continuous volume form $dV$ with finite total volume \cite[Section 3.1]{bbegz}.

\subsubsection{K-semistability}

We briefly recall the notion of K-semistability (see \cite{do1,r-t,w,od1}
for more background). A polarized complex projective variety $(X,L)$
is said to be \emph{K-semistable} if the Donaldson-Futaki invariant
$\text{DF}(\mathscr{X},\mathscr{L})$ of any test configuration $(\mathscr{X},\mathscr{L})$
for $(X,L)$ is non-negative. A test configuration $(\mathscr{X},\mathscr{L})$
is defined as a $\C^{*}-$equivariant normal model for $(X,L)$ over
the complex affine line $\C.$ More precisely, $\mathscr{X}$ is a
normal complex variety endowed with a $\C^{*}-$action $\rho$, a
$\C^{*}-$equivariant holomorphic projection $\pi$ to $\C$ and a
relatively ample $\C^{*}-$equivariant $\Q-$line bundle $\mathscr{L}$
(endowed with a lift of $\rho):$ 
\begin{equation}
\pi:\mathcal{\mathscr{X}}\rightarrow\C,\,\,\,\,\,\mathscr{L}\rightarrow\mathscr{X},\,\,\,\,\,\,\rho:\,\,\mathscr{X}\times\C^{*}\rightarrow\mathscr{X}\label{eq:def of pi for test c}
\end{equation}
such that the fiber of $\mathscr{X}$ over $1\in\C$ is equal to $(X,L).$
Its \emph{Donaldson-Futaki invariant} $\text{DF}(\mathscr{X},\mathscr{L})\in\R$
may be defined as a normalized limit, as $k\rightarrow\infty,$ of
Chow weights of a sequence of one-parameter subgroups of $GL\left(H^{0}(X,kL)\right)$
induced by $(\mathscr{X},\mathscr{L})$ (in the sense of Geometric
Invariant Theory). As a consequence, $(X,L)$ is K-semistable if,
for example, $(X,kL)$ is Chow semi-stable, for $k$ sufficiently
large \cite{r-t}. However, for the purpose of the present paper it
will be more convenient to employ the intersection-theoretic formula
for $\text{DF}(\mathscr{X},\mathscr{L})$ established in \cite{w,od1}: 

\[
\text{DF}(\mathscr{X},\mathscr{L})=\frac{a}{(n+1)!}\overline{\mathscr{L}}^{n+1}+\frac{1}{n!}\mathscr{K}_{\mathcal{\mathscr{\overline{X}}}/\P^{1}}\cdot\mathcal{\overline{\mathscr{L}}}^{n},\,\,\,\,a=-n(K_{X}\cdot L^{n-1})/L^{n}
\]
 where $\overline{\mathscr{L}}$ denotes the $\C^{*}-$equivariant
extension of $\mathscr{L}$ to the $\C^{*}-$equivariant compactification
$\mathscr{\overline{X}}$ of $\mathscr{X}$ over $\P^{1}$ and $\mathscr{K}_{\mathcal{\mathscr{\overline{X}}}/\P^{1}}$
denotes the relative canonical divisor. 
\begin{rem}
Usually the definition of $\text{DF}(\mathscr{X},\mathscr{L})$ involves
a factor of $1/L^{n},$ but the present definition will be more convenient
here (since the factor $L^{n}$ is positive it does not alter the
definition of K-stability). It is made so that $\text{DF}(\mathscr{X},\mathscr{L})=\overline{\mathscr{L}}^{n+1}$
when $\mathscr{L}=-\mathscr{K}_{\mathcal{\mathscr{\overline{X}}}/\P^{1}}.$
\end{rem}

\subsection{Arithmetic setup}

Let$\mathcal{X}$ be a projective flat scheme $\mathcal{X}\rightarrow\text{Spec \ensuremath{\Z}}$
of relative dimension $n,$ with the property that $\mathcal{X}$
is reduced and satisfies Serre's conditions $S_{2}$ (this is, for
example, the case if $\mathcal{X}$ is normal). Denote by $X$ the
complex points of $\mathcal{X}$ and assume that $X$ is a normal
projective variety over $\C.$ Such a scheme $\mathcal{X}$ will be
called an \emph{arithmetic variety. }A \emph{polarized arithmetic
variety} $(\mathcal{X},\mathcal{L}$) is an arithmetic variety endowed
with a relatively ample $\Q-$line bundle $\mathcal{L}.$ We will
denote by $L$ the ample line bundle over $X$ induced by $\mathcal{L};$
the polarized arithmetic variety $(\mathcal{X},\mathcal{L})$ will
be valled a \emph{model} for $(X,L)$\emph{ over $\Z$} (or an\emph{
integral model} for $(X,L$))). We will use the following simple
\begin{lem}
\label{lem:Stein}Under the assumptions above on $\mathcal{X}$ the
canonical embedding of $\Z$ in $H^{0}(\mathcal{X},\mathcal{O}_{\mathcal{X}})$
is an isomorphism. In other words, $1$ generates the $\Z-$module
$H^{0}(\mathcal{X},\mathcal{O}_{\mathcal{X}}).$
\end{lem}

\begin{proof}
We have injections $\Z\hookrightarrow H^{0}(\mathcal{X},\mathcal{O}_{\mathcal{X}})\hookrightarrow H^{0}(X_{\Q},\mathcal{O}_{X_{\Q}})\simeq\Q$
(using flatness in the second injection and, in the isomorphism, that
$X_{\Q}$ is geometrically connected and geometrically reduced). But,
since $H^{0}(\mathcal{X},\mathcal{O}_{\mathcal{X}})$ is a finitely
generated $\Z-$module and $\Z$ is an integrally closed domain this
implies that $\Z\hookrightarrow H^{0}(\mathcal{X},\mathcal{O}_{\mathcal{X}})$
is an isomorphism.
\end{proof}
For any positive integer $k$ we may identify the free $\Z-$module
$H^{0}(\mathcal{X},k\mathcal{L})$ with a lattice in $H^{0}(X,kL):$
\[
H^{0}(\mathcal{X},k\mathcal{L})\otimes\C=H^{0}(X,kL).
\]
By definition a\emph{ metrized line bundle} $\overline{\mathcal{L}}$
is a line bundle $\mathcal{L}\rightarrow\mathcal{X}$ such that the
corresponding line bundle $L\rightarrow X$ is endowed with a metric
$\left\Vert \cdot\right\Vert .$ We will use the additive notation
$\phi$ for metrics $\left\Vert \cdot\right\Vert $ on $L$ discussed
in the previous section: 
\[
\overline{\mathcal{L}}:=\left(\mathcal{L},\phi\right).
\]

\subsubsection{\label{subsec:Arithmetic-Fano-varieties}Arithmetic Fano varieties}

We will say that the relative canonical line bundle of an arithmetic
variety $\mathcal{X}$ is defined as a $\Q-$line bundle, denoted
by $\mathcal{K},$ if there exists a positive integer $m$ such that
the $m$th reflexive power $\omega_{X/\text{Spec}\ensuremath{\Z}}^{[m]}$
of the dualizing sheaf $\omega_{X/\text{Spec}\ensuremath{\Z}}$ of
$\mathcal{X}$ is locally free. Then the line bundle $m\mathcal{K}$
over $\mathcal{X}$ may be identified with $\omega_{X/\text{Spec}\ensuremath{\Z}}^{[m]}$
(see \cite[Section 1.1]{ko} for a more general setup of canonical
line bundles attached to schemes over regular excellent rings). An
arithmetic variety $\mathcal{X}\rightarrow\text{Spec \ensuremath{\Z}}$
will be called an \emph{arithmetic Fano variety} if
\begin{itemize}
\item the canonical line bundle $\mathcal{K}$ of $\mathcal{X}$ is well-defined
as a $\Q-$line bundle and its dual $-\mathcal{K}$ is relatively
ample
\item the complexification $X$ of $\mathcal{X}$ is normal and thus defines
a complex Fano variety (i.e. $-K_{X}$ is ample)
\end{itemize}
\begin{example}
If $\mathcal{X}$ is, locally, a complete intersection, then $\mathcal{K}$
is defined as a line bundle (i.e. $m=1)$ \cite[Section 1.1]{ko}.
In particular, if $\mathcal{X}$ is the subscheme of $\P_{\Z}^{n+1}$
cut out by an irreducible homogeneous polynomial of degree $d$ with
integer coefficents, then $\mathcal{K}$ is well-defined as a line
bundle and $\mathcal{X}$ is an arithmetic Fano variety iff $d\leq n+1.$
\end{example}

\subsubsection{\label{subsec:The-arithmetic-volume}The $\chi-$arithmetic volume,
heights and arithmetic intersection numbers}

In the arithmetic setup there are different analogs of the volume
$\text{vol\ensuremath{(L)} }$of an ample line bundle $L.$ Here we
shall focus on the one defined by the following asymptotic arithmetic
Euler characteristic originating in \cite{fa0} (called the \emph{$\chi-$arithmetic
volume} \cite{b-g-p-s,b-n-p-s} and the \emph{sectional capacity }in
\cite{r-l-v}):

\begin{equation}
\widehat{\text{vol}}_{\chi}\left(\overline{\mathcal{L}}\right):=\lim_{k\rightarrow\infty}k^{-(n+1)}\log\text{Vol}\text{\ensuremath{\left\{  s_{k}\in H^{0}(\mathcal{X},k\mathcal{L})\otimes\R:\,\,\,\sup_{X}\left\Vert s_{k}\right\Vert _{\phi}\leq1\right\} } , }\label{eq:def of xhi vol}
\end{equation}
where the volume is computed with respect to the Lebesgue measure,
normalized such that a fundamental domain of the lattice $H^{0}(\mathcal{X},k\mathcal{L})$
has unit volume. Here $H^{0}(\mathcal{X},k\mathcal{L})\otimes\R$
may be identified with the subspace of real sections in $H^{0}(X,kL).$
If the metric on $L$ has positive curvature current, then, by the
arithmetic Hilbert-Samuel theorem \cite{g-s2,Zh0},
\begin{equation}
\widehat{\text{vol}}_{\chi}\left(\overline{\mathcal{L}}\right)=\frac{\overline{\mathcal{L}}^{n+1}}{(n+1)!},\label{eq:vol chi as intersection}
\end{equation}
 where $\overline{\mathcal{L}}^{n+1}$ denotes the top arithmetic
intersection number in the sense of Gillet-Soulé \cite{g-s}, which,
defines the \emph{height} of $\mathcal{X}$ with respect to $\overline{\mathcal{L}}$
\cite{fa,b-g-s}. For the purpose of the present paper formula \ref{eq:def of xhi vol}
may be taken as the definition of $\overline{\mathcal{L}}^{n+1}$
(arithmetic intersections between general $n+1$ metrized line bundles
could then be defined by polarization). More generally, $\widehat{\text{vol}}_{\chi}\left(\overline{\mathcal{L}}\right)$
is naturally defined for $\Q-$line bundles, since it is homogeneous
with respect to tensor products of $\overline{\mathcal{L}}:$ 
\begin{equation}
\widehat{\text{vol}}_{\chi}\left(m\overline{\mathcal{L}}\right)=m^{n+1}\widehat{\text{vol}}_{\chi}\left(\overline{\mathcal{L}}\right),\,\,\,\text{if \ensuremath{m\in\Z_{+}}}\label{eq:vol chi tensor}
\end{equation}
Moreover, $\widehat{\text{vol}}_{\chi}\left(\overline{\mathcal{L}}\right)$
is additively equivariant with respect to scalings of the metric:
\begin{equation}
\widehat{\text{vol}}_{\chi}\left(\mathcal{L},\phi+\lambda\right)=\widehat{\text{vol}}_{\chi}\left(\overline{\mathcal{L}}\right)+\frac{\lambda}{2}\text{vol\ensuremath{(L)}},\,\,\,\text{if }\lambda\in\R,\label{eq:scaling of chi vol}
\end{equation}
 as follows directly from the definition.

\subsection{Upper bounds on the $\chi-$arithmetic volume vs K-semistability
of Fano varieties}

We are now ready to prove the following theorem, relating upper bounds
on the $\chi-$arithmetic volume of a metrized integral model of $(X,-K_{X})$
to K-semistability: 
\begin{thm}
\label{thm:arithm Vol and K semi st}Let $(\mathcal{X},\mathcal{L})$
be a polarized arithmetic variety such that $X$ is a Fano variety
and $L=-K_{X}.$ Then the following is equivalent:
\begin{enumerate}
\item $(X,-K_{X})$ is K-semistable
\item The supremum of $\widehat{\text{vol}}_{\chi}\left(\mathcal{L},\phi\right)$
over all continuous volume-normalized metrics $\phi$ on $-K_{X}$
is finite.
\item The supremum of $\widehat{\text{vol}}_{\chi}\left(\mathcal{L},\phi\right)$
over all continuous volume-normalized metrics $\phi$ on $-K_{X},$
which are invariant under complex conjugation, is finite.
\end{enumerate}
Moreover, $(X,-K_{X})$ is K-polystable iff the supremum in item $2$
above is attained at some locally bounded metric $\psi$ in $PSH\ensuremath{(-K_{X})}.$
In general, a locally bounded metric $\psi$ in $PSH\ensuremath{(-K_{X})}$
attains the supremum in item $2$ above iff it is a Kähler-Einstein
metric. 
\end{thm}

Recall that on any complex projective variety $X$ which is defined
over $\R$ there is a globally defined complex conjugation map (whose
orbits on $X$ correspond to the maximal ideals of the scheme $X_{\R}$$)$
and in Arakelov geometry it is often assumed that the metrics are
invariant under complex conjugation \cite{sou00}. 

Before embarking on the proof we recall the definition of the (normalized)\emph{
Ding functional} on $\mathcal{C}^{0}(-K_{X})\Cap\text{PSH\ensuremath{(-K_{X})}},$
introduced in \cite{di}, which depends on the choice of a reference
metric $\psi_{0}$ in $\mathcal{C}^{0}(-K_{X})\Cap\text{PSH\ensuremath{(-K_{X})}: }$
\begin{equation}
\mathcal{\hat{D}}_{\psi_{0}}(\psi):=-\frac{1}{\mathrm{vol}(-K_{X})}\mathcal{E}_{\psi_{0}}(\psi)-\log\int_{X}e^{-\psi},\label{eq:def of Ding f}
\end{equation}
 where the functional $\mathcal{E}_{\psi_{0}}$ is a primitive of
$(dd^{c}\psi)^{n}/n!$ (see formula \ref{eq:differential of E beauti}).
More generally, as shown in \cite{bbegz} using the monotinicity of
$\mathcal{E}_{\psi_{0}}$, $\mathcal{\hat{D}}_{\psi_{0}}(\psi)$ can
be extended to the space $\mathcal{E}^{1}(-K_{X})$ of all metrics
in $\text{PSH}\ensuremath{(-K_{X})}$ with finite energy and a finite
energy metric $\psi$ minimizes $\mathcal{\hat{D}}_{\psi_{0}}(\psi)$
iff $\psi$ is a Kähler-Einstein metric, i.e. $dd^{c}\psi$ defines
a Kähler metric on the regular locus of $X$ with constant positive
Ricci curvature. When $\psi$ is volume-normalized this equivalently
means that
\[
\frac{(dd^{c}\psi)^{n}}{\text{vol}\ensuremath{(-K_{X})}n!}=e^{-\psi}
\]
on the regular locus of $X.$ The identity \ref{eq:vol chi as intersection}
was extended to finite energy metrics in \cite{ber-f}. But for our
purposes it will be enough to work with continuous metrics.
\begin{rem}
\label{rem:KE cont}In general, any Kähler-Einstein metric $\psi$
in $\mathcal{E}^{1}(-K_{X})$ is locally bounded \cite{bbegz}. In
the toric case this implies that $\psi$ is, in fact, continuous \cite[Prop 4.1]{c-g-s}.
\end{rem}

By introducing an arithmetic version of the Ding functional we show
that item $2$ in the previous theorem is equivalent to  the normalized
Ding functional $\mathcal{\hat{D}}_{\psi_{0}}$ being bounded from
below on $\mathcal{C}^{0}(-K_{X})\Cap\text{PSH\ensuremath{(-K_{X})} }$(which
is equivalent to lower boundedness of the Mabuchi functional; see
\ref{eq:inf is inf given reference}). By \cite{li} this is equivalent
to K-semistability when $X$ is non-singular. In the proof of Theorem
\ref{thm:arithm Vol and K semi st} we explain how to extend this
result to general Fano varieties, leveraging the very recent solution
of the Yau-Tian-Donaldson conjecture for singular Fano varieties \cite{li1,l-x-z}.
The equivalence with item $3$ leverages the recent result \cite{zhu}.

\subsubsection{Proof of Theorem \ref{thm:arithm Vol and K semi st}}

We start with two lemmas. First, to a given continuous metric $\phi$
on $L$ we associate, following \cite{b-b}, a continuous psh metric
$\psi$ on $L$ defined as the following point-wise envelope:

\begin{equation}
P\phi:=\sup\left\{ \psi:\,\,\,\psi\,\text{psh},\,\,\psi\leq\phi\right\} .\label{eq:def of Pphi}
\end{equation}

\begin{rem}
More generally, when $L$ is big the envelope above has to be replaced
by its upper semi-continuous regularization in order to obtain a psh
metric. However, when $L$ is an ample line bundle over a normal variety
$X,$ as we assume here, the envelope $P\phi$ is already continuous
(see \cite[Lemma 7.9]{b-e}).
\end{rem}

\begin{lem}
\label{lem:arith vol in terms of pphi}Assume that $\mathcal{L}$
is relatively ample and let $\phi$ be a continuous metric on $L$.
Then the arithmetic $\chi-$volume may be expressed as the following
top arithmetic intersection number:
\[
\widehat{\text{vol}}_{\chi}\left(\mathcal{L},\phi\right)=\frac{\left(\mathcal{L},P\phi\right)^{n+1}}{(n+1)!}.
\]
 
\end{lem}

\begin{proof}
When $\phi$ is psh the lemma follows directly from \cite[Thm 1.4]{Zh0}
(the latter proof reduces to the original arithmetic Hilbert-Samuel
theorem in \cite{g-s2}, where $X$ is assumed non-singular, using
a perturbation argument on a resolution of $X).$ In fact, the result
\cite[Thm 1.4]{Zh0} applies more generally when $\mathcal{L}$ is
merely assumed to be relatively nef over the closed points of $\text{Spec \ensuremath{\Z} }$.
Next, the general case follows from the case when $\phi$ is psh (applied
to $P\phi)$ by the following simple observation: 
\[
\sup_{X}\left\Vert s\right\Vert _{\phi}=\sup_{X}\left\Vert s\right\Vert _{P\phi},\,\,\,\,\text{if }s\in H^{0}(X,kL),
\]
 as follows directly from the definition \ref{eq:def of Pphi} of
$P\phi$ (see \cite[Prop 1.8]{b-b}).
\end{proof}
In order to state the next lemma consider the following functional
on $\mathcal{C}^{0}(L)\cap\text{PSH \ensuremath{(L),}}$ defined with
respect to a given reference $\psi_{0}\in\mathcal{C}^{0}(L)\cap\text{PSH\ensuremath{(L):}}$
\begin{equation}
\mathcal{E}_{\psi_{0}}(\psi):=\frac{1}{(n+1)!}\int_{X}(\psi-\psi_{0})\sum_{j=0}^{n}(dd^{c}\psi)^{j}\wedge(dd^{c}\psi_{0})^{n-j}\label{eq:def of E beauti}
\end{equation}
 Alternatively, the functional $\mathcal{E}_{\psi_{0}}$ may be characterized
as the primitive of the one-form on $\mathcal{C}^{0}(L)\cap\text{PSH\ensuremath{(L)}}$
defined by the measure $(dd^{c}\psi)^{n}/n!$:

\begin{equation}
d\mathcal{E}_{\psi_{0}}(\psi)=\frac{1}{n!}(dd^{c}\psi)^{n},\,\,\,\,\mathcal{E}_{\psi_{0}}(\psi_{0})=0.\label{eq:differential of E beauti}
\end{equation}
 It follows directly from the definition of $\mathcal{E}_{\psi_{0}}(\psi)$
and the classical Hilbert-Samuel formula \ref{eq:HS formully algebraic}
that 
\begin{equation}
\mathcal{E}_{\psi_{0}}(\psi+c)=\mathcal{E}_{\psi_{0}}(\psi)+c\text{vol\ensuremath{(L),\,\,\,\forall c\in\R} }.\label{eq:scaling of beauti E}
\end{equation}
The following lemma is an arithmetic refinement of the previous formula:
\begin{lem}
\label{lem:change of metric}(change of metrics formula). For any
two continuous metrics on $L,$ which are invariant under complex
conjugation,
\begin{equation}
\widehat{\mathrm{vol}}_{\chi}\left(\mathcal{L},\phi_{1}\right)-\widehat{\mathrm{vol}}_{\chi}\left(\mathcal{L},\phi_{2}\right)=\frac{1}{2}\left(\mathcal{E}_{\psi_{0}}(P\phi_{1})-\mathcal{E}_{\psi_{0}}(P\phi_{2})\right)\label{eq:change of metric formula for chi vol}
\end{equation}
\end{lem}

\begin{proof}
When $\phi_{i}$ are psh this is well-known and follows from basic
properties of arithmetic intersection numbers; see formula \ref{eq:arithm inters form for trivial}
or \cite[Prop 2.2]{o}). Alternatively, the result follows from the
previous lemma combined with \cite[Thm A]{b-b}. In order to check
that the multiplicative normalizations adopted here are compatible
note that the scaling relations \ref{eq:scaling of chi vol} and \ref{eq:scaling of beauti E}
are indeed compatible.
\end{proof}

\subsubsection{Conclusion of the proof of Theorem \ref{thm:arithm Vol and K semi st}}

Consider the following functional on the space $\mathcal{C}^{0}(-K_{X})$
of continuous metrics on $-K_{X}$
\begin{equation}
\mathcal{\hat{D}}_{\Z}(\phi):=-2\frac{\widehat{\text{vol}}_{\chi}\left(\mathcal{L},\phi\right)}{\text{vol}(-K_{X})}-\log\int_{X}e^{-\phi}.\label{eq:def of Ding Z}
\end{equation}
 Since this functional is invariant under scalings of the metric,
$\phi\mapsto\phi+c,$ the finiteness statement in the second point
of the proposition amounts to showing that the infimum of $\mathcal{\hat{D}}_{\Z}(\phi)$
over $\mathcal{C}^{0}(-K_{X})$ is finite. Now fix a continuous psh
metric $\psi_{0}$ on $-K_{X}$ and consider the following extension
of the normalized Ding functional \ref{eq:def of Ding f} to all of
$\mathcal{C}^{0}(-K_{X}):$

\begin{equation}
\mathcal{\hat{D}}_{\psi_{0}}(\phi):=-\frac{1}{\mathrm{vol}(-K_{X})}\mathcal{E}_{\psi_{0}}(P\phi)-\log\int_{X}e^{-\phi}.\label{eq:def of Ding on non psh}
\end{equation}
Combining the previous two lemmas reveals that 
\begin{equation}
\mathcal{\hat{D}}_{\Z}(\phi)=\mathcal{\hat{D}}_{\psi_{0}}(\phi)+C_{0},\,\,\,C_{0}:=-\frac{2\left(\mathcal{L},\psi_{0}\right)^{n+1}}{\text{vol}(-K_{X})(n+1)!}\label{eq:naive arithm Ding in terms of Ding}
\end{equation}
Next, observe that 
\begin{equation}
\inf_{\mathcal{C}^{0}(-K_{X})}\mathcal{\hat{D}}_{\psi_{0}}=\inf_{\mathcal{C}^{0}(-K_{X})\Cap\text{PSH\ensuremath{(-K_{X})}}}\mathcal{\hat{D}}_{\psi_{0}}\label{eq:inf Ding is inf Ding without psh}
\end{equation}
Indeed, this follows directly from the fact that the operator $\phi\mapsto P\phi$
from $\mathcal{C}^{0}(L)$ to $\mathcal{C}^{0}(L)\Cap\text{PSH\ensuremath{(L)}}$
is increasing and satisfies $P^{2}=P.$ 

\subsubsection*{$"3"\protect\implies"1"$ }

Let us first recall how Item $2$ implies Item $1$. First Item $2$
implies, thanks to the identities \ref{eq:naive arithm Ding in terms of Ding}
and \ref{eq:inf Ding is inf Ding without psh}, that the infimum of
$\mathcal{\hat{D}}_{\psi_{0}}$ over $\mathcal{C}^{0}(-K_{X})\Cap\text{PSH\ensuremath{(-K_{X})}}$
is finite. Thus it follows from results in \cite{ber0} that $(X,-K_{X})$
is K-semistable. Let us next show how to refine the proof in \cite{ber0}
to show the stronger statement $"3"\implies"1".$ More generally,
we will show that when $X$ is defined over the real field $\R$ $X$
is K-semistable if the infimum of $\mathcal{\hat{D}}_{\psi_{0}}$
over the space $\overline{\mathcal{C}^{0}(-K_{X})}\Cap\text{PSH\ensuremath{(-K_{X})}}$
is finite, where $\overline{\mathcal{C}^{0}(L)}$ denotes the subspace
of $\mathcal{C}^{0}(L)$ consisting of metrics which are invariant
under complex conjugation. To this end let us first summarize the
main steps in the proof in \cite{ber0}. First, a test configuration
$(\mathscr{X},\mathscr{L})$ for $(X,-K_{X})$ and a given metric
$\phi$ for $-K_{X}$ in $\mathcal{C}^{0}(-K_{X})\Cap\text{PSH}\ensuremath{(-K_{X})}$
determines a ray $\phi_{t}$ in $\text{PSH\ensuremath{(-K_{X})}}$
emanating from $\phi$ parametrized by $t\in[0,\infty[$ (i.e. $\phi_{0}=\phi).$
Using the notation in formula \ref{eq:def of pi for test c} the ray
$\phi_{t}$ is defined by 
\[
\phi_{-\log|\tau|}=\rho(\tau)^{*}(\Phi_{|\mathscr{X}_{\tau}}),\,\,\,\tau\in\C^{*}
\]
 where $\Phi$ is the $S^{1}-$invariant metric on the restriction
of $\mathcal{L}$ to the inverse image $\pi^{-1}(\mathbb{D})$ in
$\mathcal{X}$ of the unit-disc $\mathbb{D}\subset\C$ defined by
\begin{equation}
\Phi:=\sup\left\{ \Psi:\,\,\Psi_{|\pi^{-1}(\partial\mathbb{D})}=\phi,\,\,\,\,\Psi\in\mathcal{C}^{0}(\mathcal{L})\Cap\text{PSH}(\mathcal{L}_{|\pi^{-1}(\mathbb{D})})\right\} ,\label{eq:def of Phi as envelope}
\end{equation}
 where we have used the $\C^{*}-$action $\rho$ to identify $X$
with $X_{\tau}$ for any $\tau$ in the unit-circle $\partial\mathbb{D}.$
By \cite[Thm 1.3]{ber0}
\[
\text{vol \ensuremath{(-K_{X})^{-1}}}\text{DF (\ensuremath{\mathscr{X}},\ensuremath{\mathscr{L}})\ensuremath{\geq\lim_{t\rightarrow\infty}\left(t^{-1}\mathcal{\hat{D}}_{\phi_{0}}(\phi_{t})\right)}}.
\]
When $\mathcal{\hat{D}}_{\phi_{0}}(\phi_{t})$ is bounded from below
this means that $\text{DF}(\mathscr{X},\mathscr{L})\geq0,$ showing
that $X$ is K-semistable. Now assume that $X$ is defined over the
real field $\R.$ Then it follows from \cite[Thm 1.1]{zhu} that in
order to check K-semistability of $(X,-K_{X})$ it is enough to consider
test configurations $(\mathscr{X},\mathscr{L})$ defined over $\R.$
Thus, we just have to verify that for such test configurations, if
the given metric $\phi$ is taken to be in $\overline{\mathcal{C}^{0}(-K_{X})}\Cap\text{PSH\ensuremath{(-K_{X})}},$
then the ray $\phi_{t}$ remains in $\overline{\mathcal{C}^{0}(-K_{X})}\Cap\text{PSH\ensuremath{(-K_{X})}},$
for all $t>0.$ Since $(\mathscr{X},\mathscr{L})$ is defined over
$\R$ there is a complex conjugation map $F$ from $\mathscr{X}$
to $\mathscr{X}$ (that lifts to $\mathscr{L})$ and thus it is enough
to show that $F^{*}\phi=\phi$ implies that $F^{*}\Phi=\Phi.$ But
this follows from the definition \ref{eq:def of Phi as envelope}
of $\Phi$ only using that $F^{*}$ preserves the psh property of
a metric (as follows from a direct local calculation that reduces
to the fact that the Laplacian $i\partial_{z}\partial_{\bar{z}}$
in $\C$ is invariant under $z\mapsto\bar{z}).$ 

\subsubsection*{$"1"\protect\implies"2"$}

First recall that any K-semistable normal Fano variety (i.e. such
that $(X,-K_{X})$ is K-semistable) has log terminal singularities
\cite[Thm 1.3]{od0}. In the case that $X$ is non-singular it was
shown in \cite{li} that if $X$ is K-semistable, then the infimum
of the Ding functional $\mathcal{\hat{D}}_{\psi_{0}}$ over $\mathcal{C}^{0}(-K_{X})\Cap\text{PSH}\ensuremath{(-K_{X})}$
is finite. Thus, by formula \ref{eq:inf Ding is inf Ding without psh},
so is the infimum of $\mathcal{\hat{D}}_{\psi_{0}}$ over $\mathcal{C}^{0}(-K_{X}).$
The proof in \cite{li} relied, in particular, on the resolution of
the Yau-Tian-Donaldson conjecture in \cite{c-d-s} for Fano manifolds.
But thanks to the recent resolution of the Yau-Tian-Donaldson conjecture
for singular Fano varieties the proof in \cite{li} can be extended
to singular Fano varieties, mutatis mutandis. We briefly summarize
the argument, using Deligne pairings as in \cite{ber0} (rather than
the Bott-Chern classes used in \cite{li}). The starting point is
the result \cite[Thm 1.3]{l-w-x}, saying that if $X$ is K-semistable
then there exists a test configuration $(\mathscr{X},\mathscr{L})$
for $(X,-K_{X})$ whose central fiber $X_{0}$ is given by a K-polystable
Fano variety. More precisely, the test configuration is \emph{special}
in the sense that $\mathscr{L}$ is the relative anti-canonical line
bundle. Since the central fiber $X_{0}$ of $\mathscr{X}$ is K-polystable
it admits, by the solution of the Yau-Tian-Donaldson conjecture for
singular Fano varieties \cite{l-x-z} (building on \cite{li1}) a
Kähler-Einstein metric $\phi_{KE}.$ It thus follows from \cite[Thm 4.8]{bbegz}
that the Ding functional is bounded from below on $\mathcal{C}^{0}(-K_{X_{0}})\Cap\text{PSH}(-K_{X_{0}})$.
More precisely, its infimum is attained at the Kähler-Einstein metric
$\phi_{KE}:$ 
\begin{equation}
\inf_{\mathcal{C}^{0}(-K_{X_{0}})\Cap\text{PSH}(-K_{X_{0}})}\mathcal{\hat{D}}=\mathcal{\hat{D}}(\phi_{KE})>-\infty.\label{eq:Ding minimal on KE on central fiber}
\end{equation}
Now, given a metric $\phi$ in $\mathcal{C}^{0}(-K_{X})\Cap\text{PSH\ensuremath{(-K_{X})} let}$
$\Phi$ be the corresponding metric on $\mathscr{L}\rightarrow\pi^{-1}(\mathbb{D})$
defined by formula \ref{eq:def of Phi as envelope}. It induces a
metric on the $(n+1)-$fold Deligne pairing $\left\langle \mathscr{L},\mathscr{L},...,\mathscr{L}\right\rangle \rightarrow\mathbb{D}$
that we denote by $\left\langle \Phi\right\rangle $ (see \cite[Section 2.3]{ber0}).
Consider the corresponding twisted metric on $-\left\langle \mathscr{L},\mathscr{L},...,\mathscr{L}\right\rangle \rightarrow\mathbb{D}$
defined by 
\[
-\left\langle \Phi\right\rangle -\log\int_{X_{\tau}}e^{-\Phi_{|X_{\tau}}},
\]
dubbed the \emph{Ding metric }in \cite{ber0}. Fixing a trivialization
$S(\tau)$ of $\left\langle \mathscr{L},\mathscr{L},...,\mathscr{L}\right\rangle \rightarrow\mathbb{D}$
we may identify this metric with a function $\psi(\tau)$ on $\mathbb{D}:$
\[
\psi(\tau):=\log\left(\left\Vert S(\tau)\right\Vert _{\left\langle \Phi\right\rangle }^{2}\right)-\log\int_{X_{\tau}}e^{-\Phi_{|X_{\tau}}},
\]
For a fixed $\tau$ this metric coincides with the normalized Ding
functional $\mathcal{\hat{D}}(\phi_{\tau})$ up to an additive constant
depending on $\tau$ (by the ``change of metrics formula'' for Deligne
pairing; see \cite[Section 2.3]{ber0}). In particular, there exists
$a\in\R$ such that

\begin{equation}
\psi(1):=\mathcal{\hat{D}}_{\psi_{0}}(\phi)+a,\,\,\,\,\,\,\psi(0)\geq b:=\log\left(\left\Vert S(0)\right\Vert _{\left\langle \phi_{KE}\right\rangle }^{2}\right)-\log\int_{X_{0}}e^{-\phi_{KE}},\label{eq:psi 1 and 2}
\end{equation}
 using \ref{eq:Ding minimal on KE on central fiber} in the inequality.
As shown in \cite[Prop 3.5]{ber0} $\psi(\tau)$ is subharmonic on
$\mathbb{D}$ and the first term $\left\langle \Phi\right\rangle $
is continuous on $\mathbb{D}$ (as follows from \cite[Thm A]{Mo};
see the proof of \cite[Prop 3.6]{ber0}). Moreover, the second term
is also continuous on $\mathbb{D},$ as shown when $X$ is non-singular
in \cite[Lemma 1.9]{li} and in general in \cite[Lemma 7.1]{l-w-x0}.
As a consequence, 
\[
\psi(0)\leq\int_{\partial D}\psi d\theta=\psi(1),
\]
 using that $\psi(\tau)$ is $S^{1}-$invariant in the last equality.
Finally, invoking formula \ref{eq:psi 1 and 2} shows that $\mathcal{\hat{D}}_{\psi_{0}}(\phi)$
is uniformly bounded from below, as desired.

\subsection{\label{subsec:Compatibility-of-Conjecture}Compatibility of Conjecture
\ref{conj:height intro} with taking products}

The previous theorem shows, in particular, that the K-semistability
assumption in Conjecture \ref{conj:height intro} is necessary. We
next show that the conjecture is compatible with taking products:
\begin{prop}
\label{prop:product}Let $m\geq2$ and $\mathcal{X}_{1},...,\mathcal{X}_{m}$
be arithmetic Fano varieties which are K-semistable over $\C.$ Assume
that the inequality in Conjecture \ref{conj:height intro} holds for
all $\mathcal{X}_{i}$ (for any volume-normalized metrics on $-K_{X_{i}}$
with positive curvature current). Then the inequality holds for $\mathcal{X}_{1}\times\cdots\times\mathcal{X}_{m}$
with \emph{strict }inequality (for any volume-metric normalized metric
on $-K_{X_{1}\times\cdots\times X_{m}}$ with positive curvature current).
More precisely, 
\[
\widehat{\text{vol}}_{\chi}(\overline{-\mathcal{K}_{\mathcal{X}_{1}\times\cdots\times\mathcal{X}_{m}}})<\widehat{\text{vol}}_{\chi}(\overline{-\mathcal{K}_{\P_{\Z}^{n}}}).
\]
\end{prop}

\begin{proof}
By a simple induction argument it is enough to consider the case when
$m=2.$ First note that, in general, given two polarized metrized
arithmetic varieties $(\mathcal{X}_{i},\overline{\mathcal{L}_{i}})$
of relative dimension $n_{i}$
\begin{equation}
\frac{\widehat{\text{vol}}_{\chi}\left(\rho_{1}^{*}\overline{\mathcal{L}_{1}}\otimes\rho_{2}^{*}\overline{\mathcal{L}_{2}}\right)}{\text{vol}\left(\rho_{1}^{*}\overline{\mathcal{L}_{1}}\otimes\rho_{2}^{*}\overline{\mathcal{L}_{2}}\right)}=\frac{\widehat{\text{vol}}_{\chi}\left(\overline{\mathcal{L}_{1}}\right)}{\text{vol}\left(\overline{\mathcal{L}_{1}}\right)}+\frac{\widehat{\text{vol}}_{\chi}\left(\overline{\mathcal{L}_{2}}\right)}{\text{vol}\left(\overline{\mathcal{L}_{2}}\right)}\label{eq:additivity of normalized height}
\end{equation}
 where $\rho_{1}$ and $\rho_{2}$ denote the natural morphisms from
$\mathcal{X}_{1}\times\mathcal{X}_{2}$ to $\mathcal{X}_{1}$ and
$\mathcal{X}_{2},$ respectively (as follows readily from formula
\ref{eq:def of xhi vol}).

Assume now that the inequality in Conjecture \ref{conj:height intro}
holds for $\overline{-\mathcal{K}_{\mathcal{X}_{1}}}$ and $\overline{-\mathcal{K}_{\mathcal{X}_{2}}}$.
Endow $-K_{X_{1}\times X_{2}}$ with the induced product metric (which
is volume-normalized, since the metrics on $-K_{X_{i}}$ are assumed
to be volume-normalized). The identity \ref{eq:additivity of normalized height}
yields 
\[
\widehat{\text{vol}}_{\chi}\left(\overline{-\mathcal{K}_{\mathcal{X}_{1}\times\mathcal{X}_{2}}}\right)=\widehat{\text{vol}}_{\chi}\left(\overline{-\mathcal{K}_{\mathcal{X}_{1}}}\right)\text{vol}\left(-K_{X_{2}}\right)+\widehat{\text{vol}}_{\chi}\left(\overline{-\mathcal{K}_{\mathcal{X}_{1}}}\right)\text{vol}\left(-K_{X_{1}}\right).
\]
 Accordingly, by assumption, 
\[
\widehat{\text{vol}}_{\chi}\left(\overline{-\mathcal{K}_{\mathcal{X}_{1}\times\mathcal{X}_{2}}}\right)\leq\widehat{\text{vol}}_{\chi}\left(\overline{-\mathcal{K}_{\P_{\Z}^{n_{1}}}}\right)\text{vol}\left(-K_{X_{2}}\right)+\widehat{\text{vol}}_{\chi}\left(\overline{-\mathcal{K}_{\P_{\Z}^{n_{2}}}}\right)\text{vol}\left(-K_{X_{1}}\right),
\]
where the projective spaces have been induced by the volume-normalized
Fubini-Study metric and we have used that $\widehat{\text{vol}}_{\chi}\left(\overline{-\mathcal{K}_{\P_{\Z}^{n}}}\right)$
is positive for any $n$ (as shown in \ref{lem:arithm vol of proj etc}).
Hence, applying Fujita's inequality \ref{eq:fujita intro}, yields
\[
\widehat{\text{vol}}_{\chi}\left(\overline{-\mathcal{K}_{\mathcal{X}_{1}\times\mathcal{X}_{2}}}\right)\leq\widehat{\text{vol}}_{\chi}\left(\overline{-\mathcal{K}_{\P_{\Z}^{n_{1}}}}\right)\text{vol}\left(-K_{\P_{\C}^{n_{1}}}\right)+\widehat{\text{vol}}_{\chi}\left(\overline{-\mathcal{K}_{\P_{\Z}^{n_{2}}}}\right)\text{vol}\left(-K_{\P_{\C}^{n_{2}}}\right).
\]
 But, the rhs above equals $\widehat{\text{vol}}_{\chi}\left(\overline{-\mathcal{K}_{\P_{\Z}^{n_{1}}\times\P_{\Z}^{n_{2}}}}\right)$
(by the identity \ref{eq:additivity of normalized height}), which
is strictly smaller than $\widehat{\text{vol}}_{\chi}\left(\overline{-\mathcal{K}_{\P_{\Z}^{n_{1}+n_{2}}}}\right),$
by the toric case, considered in Section \ref{subsec:The-case-of-toric products}.

All that remains is thus to show that the sup of $\widehat{\text{vol}}_{\chi}(\overline{-\mathcal{K}_{\mathcal{X}_{1}\times\mathcal{X}_{2}}})$
over all continuous volume-normalized metrics coincides with the sup
restricted to the ones which have positive curvature current and are
product metrics. First, as shown in the proof of Theorem \ref{thm:arithm Vol and K semi st}
we may restrict to those with positive curvature current. To prove
that we may restrict to product metrics first consider the case when
$(X_{i},-K_{X_{i}})$ are both K-polystable. They thus admit Kähler-Einstein
metrics and the corresponding product metric is Kähler-Einstein on
$X_{1}\times X_{2}$ and, as a consequence, realizes the sup of $(\overline{-\mathcal{K}_{\mathcal{X}_{1}\times\mathcal{X}_{2}}})^{n+1},$
by Theorem \ref{thm:arithm Vol and K semi st} (strictly speaking,
in the singular case the Kähler-Einstein metric is merely known to
be locally bounded, but it can, in a standard way, be approximated
by continuous ones). Finally, in the case when $(X_{i},-K_{X_{i}})$
are merely K-semistable we will use the following general observation.
If $X_{1}$ and $X_{2}$ are K-semistable Fano varieties over $\C,$
then the inf of the Ding functional (formula \ref{eq:def of Ding f})
corresponding to $X_{1}\times X_{2}$ coincides with the inf over
product metrics. To prove this first recall the definition of the
twisted Ding normalized functional $\mathcal{\hat{D}}_{\psi_{0},\gamma}$
corresponding to a given locally bounded psh metric $\psi_{0}$ and
$\gamma\in]0,1]:$ 
\[
\mathcal{\hat{D}}_{\psi_{0},\gamma}(\psi)=-\frac{1}{\mathrm{vol}(-K_{X})}\mathcal{E}_{\psi_{0}}(\psi)-\log\int_{X}e^{-\left(\gamma\psi+(1-\gamma)\psi_{0}\right)}
\]
By Hölder's inequality $\mathcal{\hat{D}}_{\psi_{0},\gamma}(\psi)$
is decreasing in $\gamma.$ Since, as shown in the proof of Theorem
\ref{thm:arithm Vol and K semi st}, $\mathcal{\hat{D}}_{\psi_{0},1}$
is bounded from below when $X$ is K-semistable, so is $\mathcal{\hat{D}}_{\psi_{0},\gamma}(\psi)$
for any $\gamma\in]0,1[.$ More precisely, $\mathcal{\hat{D}}_{\psi_{0},\gamma}(\psi)$
is coercive for any given $\gamma\in]0,1[$ (see the proof of \cite[Cor 3.6]{berm6})
and thus $\mathcal{\hat{D}}_{\psi_{0},\gamma}$ admits a minimizer
$\psi_{\gamma}$ and the minimizers are precisely the solutions to
the twisted Kähler-Einstein equation 
\[
\frac{(dd^{c}\psi)^{n}/n!}{\text{vol\ensuremath{(-K_{X})}}}=\frac{e^{-\left(\gamma\psi+(1-\gamma)\psi_{0}\right)}}{\int_{X}e^{-\left(\gamma\psi+(1-\gamma)\psi_{0}\right)}}
\]
 (see \cite{bbegz}). Thus, given two K-semistable Fano varieties
$X_{1}$ and $X_{2}$ and $\gamma\in]0,1[$ we may take two twisted
KE-metrics $\psi_{\gamma}^{(1)}$ and $\psi_{\gamma}^{(2)}$ on $-K_{X_{1}}$
and $-K_{X_{2}}$, respectively. The corresponding product metric
$\psi_{\gamma}$ on $-K_{X_{1}\times X_{2}}$ is a twisted KE-metric
and thus minimizes the twisted normalized Ding functional $\mathcal{\hat{D}}_{\psi_{0},\gamma}$
on $X_{1}\times X_{2}.$ Moreover, as $\gamma\rightarrow1$ 
\begin{equation}
\mathcal{\hat{D}}_{\psi_{0}}(\psi_{\gamma})\rightarrow\inf\mathcal{\hat{D}}_{\psi_{0}}.\label{eq:conv towards inf}
\end{equation}
Indeed, $\gamma\rightarrow\mathcal{\hat{D}}_{\psi_{0},\gamma}(\psi)$
is continuous and concave on $]0,1]$ for a fixed continuous metric
$\psi,$ by Hölders inequality. The convergence \ref{eq:conv towards inf}
thus follows from Lemma \ref{lem:element} below. Finally, since in
our setup $\psi_{\gamma}$ is a product metric it follows that the
inf of $\mathcal{\hat{D}}_{\psi_{0}}$ coincides with  the inf restricted
to product metrics, as desired.
\end{proof}
In the proof we used the following elementary result about convex
functions (applied to $-f)$:
\begin{lem}
\label{lem:element}Let $f(t)$ be a function $[0,1]\rightarrow]-\infty,\infty]$
of the form 

\[
f(t)=\sup_{p\in\mathcal{P}}\left(f_{p}(t)\right),
\]
 where $f_{p}(t)$ is a family of continuous convex functions on $[0,1],$
parametrized by a set $\mathcal{P}.$ Then $f(t)$ is continuous on
$[0,1].$ 
\end{lem}

\begin{proof}
This is standard, but for completeness we provide a proof. First recall
that the sup of a family of continuous functions is lower semi-continuous.
Hence, it will be enough to show that $f(t)$ is upper semi-continuous
(usc). To this end first observe that since $t\mapsto f_{p}(t)$ is
convex it follows that $f(t)$ is also convex. But any convex function
on $[0,1]$ is usc. Indeed, $f$ is (Lipschitz) continuous on $]0,1[,$
since it is convex there. By symmetry, it is thus enough to prove
upper continuity at $t=1.$ Now, since $f(t)$ is convex we have,
given $t\in]0,1[,$ that
\[
f(1)\geq f(t)+(1-t)\partial f(t)
\]
 for any subgradient $\partial f(t)$ at $t,$ i.e. any one-sided
derivative at $t.$ But since $f(t)$ is convex  $\partial f(t)\geq\partial f(t_{0})$
for any fixed $t_{0}$ such that $t_{0}\leq t.$ Hence, $f(1)\geq f(t)+(1-t)\partial f(t_{0})$
and letting $t\rightarrow1$ thus shows that $f(1)$ is greater than
or equal to the limtsup of $f(t)$ as $t\rightarrow1,$ as desired. 
\end{proof}

\section{\label{sec:Sharp-height-inequalities}Sharp height inequalities in
the toric case}

We now specialize to the case when $X$ is toric Fano variety.

\subsection{The toric setup}

We start by recalling the notation for toric metrics employed in \cite{ber-ber}
and the relation to the canonical toric integral model. 

\subsubsection{The moment polytope $P(L)$}

Let $X$ be an $n-$dimensional complex projective toric variety,
i.e. a complex projective variety endowed with an action of the $n-$dimensional
complex torus $\C^{*n}$ with an open dense orbit. We shall denote
by $T_{c}$ the complex torus and by $T$ the real maximal compact
subtorus of $T_{c},$ i.e. $T=(S^{1})^{n}.$ Let $L$ be a toric ample
line bundle, i.e. an ample line bundle over $X$ endowed with a $T_{c}-$action
covering the action of $T_{c}$ on $X.$ It induces a bounded convex
polytope $P(L)$ in $\R^{d}$ with non-empty interior, defined as
follows. Consider the induced action of the group $T_{c}$ on the
space $H^{0}(X,kL)$ of global holomorphic sections of $kL\rightarrow X$
(for $k$ a given positive integer). Decomposing the action of $T_{c}$
according to the corresponding one-dimensional representations $e^{m},$
labeled by $m\in\Z^{n}:$ 
\begin{equation}
H^{0}(X,kL)=\oplus_{m\in B_{k}}\C e^{m}\label{eq:decomp of H not wrt monom}
\end{equation}
 the lattice polytope $P_{(X,L)}$ may be defined as the convex hull
of $k^{-1}B_{k}$ in $\R^{n}$, for $k$ sufficiently large. More
generally, by homogeneity, $P_{(X,L)}$ is defined for any ample $\Q-$line
bundle. 

In particular, if $X$ is Fano, then the polytope $P(-K_{X})$ has
vertices in $\Q^{n}$ and may be represented as follows:
\begin{equation}
P(-K_{X})=\left\{ p\in\R^{n}:\,\,\left\langle l_{F},p\right\rangle \geq-1,\,\,\forall F\right\} ,\label{eq:P in Fano case}
\end{equation}
where $F$ ranges over all facets of $P(-K_{X})$ and $l_{F}$ denotes
the unique primitive element in $\Z^{n}$ which is an interior normal
to the facet $F$ (i.e. $P(-K_{X})$ is the dual of the polytope with
primitive vertices $l_{F}).$ Conversely, any such polytope corresponds
to a Fano variety $X$ \cite{c-l-s,ber-ber}. 
\begin{example}
\label{exa:toric family}When $X=\P^{n}$ the polytope $P(-K_{X})$
is $(n+1)\left(\Sigma_{n}-(1,....,1)\right)$ where $\Sigma_{n}$
denotes the $n-$dimensional unit-simplex. An infinite family of two-dimensional
toric Fano varieties $X_{p,q},$ parametrized by two prime numbers
$p$ and $q,$ is obtained by letting $P(-K_{X_{p,q}})$ be the polytope
which is dual to the polytope with the four primitive vertices $(\pm p,\pm q).$
In particular, $\text{vol\ensuremath{(-K_{X_{p,q}})}}=2/(pq)$ tends
to zero when $pq$ tends to infinity. 
\end{example}

\begin{rem}
From an invariant point of view, the real vector space $\R^{n}$ above
arises as $M\otimes_{\Z}\R,$ where $M$ is the lattice $\text{Hom\ensuremath{(T_{c},\C^{*})}}$
of characters of the group $T_{c}$ (cf. \cite{c-l-s}).
\end{rem}

\subsubsection{\label{subsec:Logarithmic-coordinates-and}Logarithmic coordinates
and the Legendre transform $\phi^{*}$ of a metric $\phi$ on $L$}

Since $X$ is toric we can identify $T_{c}$ with its open orbit in
$X.$ Let $\mbox{Log }$be the map from $T_{c}$ to $\R^{n}$ defined
by 
\[
\text{Log:\,}T_{c}\rightarrow\R^{n},\,\,\,\mbox{Log}(z):=x:=(\log(|z_{1}|^{2}),...,\log(|z_{n}|^{2}).
\]
The real compact torus $T$ acts transitively on its fibers. We will
refer to $x$ as the (real) \emph{logarithmic coordinates} on $T_{c}.$
Let $L$ be a toric ample line bundle over $X$ and assume that $P$
contains the origin, $0\in P,$ and denote by $e^{0}$ the corresponding
$T-$invariant element in $H^{0}(X,kL).$ Any continuous $T-$invariant
metric $\left\Vert \cdot\right\Vert $ on $L$ induces a continuous
function on $\R^{n}$ that we shall denote by $\phi(x),$ defined
as
\[
\phi(x):=-\log\left(\left\Vert e^{0}\right\Vert ^{2}(z)\right),\,\,\,z\in T_{c}\Subset X,\,\,\,x:=\text{Log \ensuremath{z}}.
\]
Thus, in the present additive notation $\phi$ for metrics we have
$\phi(x)=\phi_{U}(z),$ when $U=T_{c},$ abusing notation slightly.
The Legendre transform of $\phi(x),$ which defines a lower-semicontinuous
convex function on $\R^{n}$ (taking values in $]-\infty,\infty])$
will be denoted by $\phi^{*}:$ 
\[
\phi^{*}(p):=\sup_{x\in\R^{n}}\left\langle p,x\right\rangle -\phi(x).
\]
A $T-$invariant continuous metric $\psi$ on $L$ is psh iff the
corresponding function $\psi(x)$ on $\R^{n}$ is convex (iff $\psi(x)=\psi^{**}(x)).$
We will denote by $\psi_{P(L)}$ the unique continuous convex function
on $\R^{n}$ whose Legendre transform is equal to $0$ on $P(L)$
and equal to $\infty$ on the complement of $P(L):$ 
\begin{equation}
\psi_{P(L)}(x):=\sup_{p\in P(L)}\left\langle p,x\right\rangle \,\,\,\,\,\,\,\,(\psi_{P(L)}^{*}=0\,\,\text{on \ensuremath{P,\,\,\psi_{P(L)}^{*}=\infty\,\,\text{on \ensuremath{P(L)^{c})}}}}\label{eq:def of psi P x}
\end{equation}
It corresponds to a continuous psh metric on $L$ (see the proof of
\cite[Prop 3.3]{ber-ber}) and it will be used as a canonical reference
metric in the present toric setup. It follows that for any other continuous
metric $\phi$ on $L$ 
\begin{equation}
\phi-\psi_{P(L)}\in L^{\infty}(\R^{n}),\,\,\,P(L)=\overline{\left\{ \phi^{*}<\infty\right\} .}\label{eq:P in terms of Legendre}
\end{equation}

\begin{rem}
From an invariant point of view the logarithm coordinates take value
in $N\otimes\R,$ where $N$ is the lattice $\text{Hom\ensuremath{(\C^{*},T_{c})}}$
of one-parameter subgroups of $T_{c},$ i.e. the dual of the lattice
$\text{Hom\ensuremath{(T_{c},\C^{*})}}$ of characters of $T_{c}.$ 
\end{rem}

\subsubsection{Pushing forward measures from $X$ to $\R^{n}$}

For any $T-$invariant continuous psh metric $\psi$ on $L$ the push-forward
of the measure $(dd^{c}\psi)^{n}/n!$ on $L$ under the map $\mbox{Log}$
is given by 
\[
\mbox{Log}\left(\frac{(dd^{c}\psi)^{n}}{n!}\right)=\det(\nabla^{2}\phi)dx,
\]
(since the integral along the $T^{n}-$fibers equals $(2\pi)^{n}).$
The measure in the right hand side is defined in the weak sense of
Alexandrov. Since the closure of the image of $\R^{n}$ under the
sub-gradient map of $\phi$ equals $P$ it follows that
\[
\text{vol}(L)=\int_{P}dy:=\text{Vol\ensuremath{(P)}},
\]
where $dy$ is Lebesgue measure. Next consider the case when $L=-K_{X}.$
Then 
\begin{equation}
e_{0}:=z_{1}\frac{\partial}{\partial z_{1}}\wedge\cdots\wedge z_{n}\frac{\partial}{\partial z_{n}}\label{eq:def of invariant section}
\end{equation}
 defines a $T_{c}-$invariant global holomorphic section of $-K_{X},$
trivializing $-K_{X}$ over $U:=\C^{*n}.$ We can thus identify a
continuous metric $\phi$ on $-K_{X}$ with the corresponding function
$\phi_{U}$ on $\C^{*n}$ (formula \ref{eq:def of phi U}) and volume
form on $X$ (formula \ref{eq:volume form e minus phi U}) expressed
as follows on $\C^{*n},$ with respect to the local holomorphic coordinate
$\log z:$

\[
e^{-\phi_{U}}(\frac{i}{2})^{n}d(\log z_{1})\wedge d(\log\overline{z}_{1})\wedge\cdots\wedge d(\log z_{n})\wedge d(\log\overline{z}_{n})
\]
 symbolically denoted by $e^{-\phi}.$ Using again that the integral
along the $T^{n}-$fibers equals $(2\pi)^{n}$ yields 
\begin{equation}
\int_{X}e^{-\phi}=\pi^{n}\int_{\R^{n}}e^{-\phi(x)}dx.\label{eq:integral of exp minus phi in toric}
\end{equation}

\subsubsection{K-semistability and toric Kähler-Einstein metrics}

We recall the following result, which is a combination of the results
\cite[Thm 1.2]{ber-ber} and \cite[Cor 1.2]{ber0} (which are formulated
in terms of $T_{c}-$equivariant K-polystability and K-polystability,
respectively). 
\begin{prop}
\label{prop:toric ke}Let $X$ be a toric Fano variety. The following
is equivalent:
\begin{itemize}
\item $X$ is K-semistable
\item $X$ is K-polystable
\item $X$ admits a $T-$invariant Kähler-Einstein metric
\item The barycenter of $P(-K_{X})$ coincides with the origin $0.$
\end{itemize}
\end{prop}

\subsubsection{The arithmetic $\chi-$volume of a toric metric }

Any toric ample line bundle $L\rightarrow X$ admits a canonical integral
model $\mathcal{L}\rightarrow\mathcal{X}$ over $\Z$ with $\mathcal{X}$
normal (see \cite[Section 2]{ma} and \cite[Def 3.5.6]{b-g-p-s}).

The following result is a special case of the main result of \cite[Thm 3]{b-g-p-s}
(combined with Lemma \ref{lem:arith vol in terms of pphi}):
\begin{prop}
\label{prop:prop arithm vol in toric case}Let $L\rightarrow X$ be
an ample toric line bundle and denote by $(\mathcal{X},\mathcal{L})$
its canonical toric model over $\Z.$ Assume that $\phi$ is a continuous
$T-$invariant metric on $L.$ Then
\[
2\widehat{\mathrm{vol}}_{\chi}\left(\mathcal{L},\phi\right)=-\int_{P(L)}\phi^{*}dy.
\]
 
\end{prop}

An alternative analytic proof of this formula can also be given, using
that the integral lattice $H^{0}(\mathcal{X},k\mathcal{L})$ in $H^{0}(X,kL)$
is generated by the $T_{c}-$equivariant bases $e^{m}$ appearing
in the decomposition \ref{eq:decomp of H not wrt monom} \cite{ma}.
Since this basis is ortonormal wrt the $L^{2}-$norm on $H^{0}(X,kL)$
induced by the metric $\psi_{P(L)}$ on $L,$ defined by formula \ref{eq:def of psi P x}
and the Haar measure on the unit-torus $T\Subset X,$ applying \cite[Thm A]{b-b}
yields 

\begin{equation}
\widehat{2\text{vol}}\left(\mathcal{L},\phi\right)=\mathcal{E}_{\psi_{P(L)}}(\phi).\label{eq:toric arithm vol as beautiful E}
\end{equation}
 When $\phi$ is toric the right hand side above coincides, by \cite[Prop 2.9]{ber-ber},
with the right hand side of the formula in the previous proposition. 

\subsubsection{Arithmetic toric Fano varieties}

Now assume that $X$ is a toric Fano variety, so that $-K_{X}$ defines
an ample $\Q-$line bundle. Then the canonical integral model $\mathcal{X}$
of $X$ over $\Z$ is a normal arithmetic Fano variety, i.e. the relative
anti-canonical divisor $-\mathcal{K}$ on $\mathcal{X}$ defines a
relatively ample $\Q-$line bundle on $\mathcal{X}.$ Indeed, $-\mathcal{K}$
coincides with the canonical integral model $\mathcal{L}$ of $-K_{X}.$
This follows (just as in the function field case considered in \cite[Lemma 2.2]{ber0})
from the fact that the fibers of the structure morphism $\mathcal{X}\rightarrow$$\text{Spec \ensuremath{\Z} }$are
reduced and irreducible. 

\subsection{\label{subsec:Proof-of-Theorem main}Proof of Theorem \ref{thm:main toric intro}}

Given a Fano variety $X,$ let $\phi$ be a continuous metric on $-K_{X}$
which is volume-normalized. We will prove the following more general
formulation of the inequality in Theorem \ref{thm:main toric intro}
(where the psh assumption on $\phi$ has been dispensed with): 

\[
\widehat{\text{vol}}_{\chi}\left(\mathcal{-}\mathcal{K},\phi\right)\leq\widehat{\text{vol}}_{\chi}\left(-\overline{\mathcal{K}_{\P_{\Z}^{n}}}\right)
\]
 where the metric on $-K_{\P^{n}}$ is the one induced by the volume-normalized
Fubini-Study metric.

A $T-$invariant continuous metric $\phi$ will, as above, be identified
with a  function $\phi(x)$ on $\R^{n}.$ If $\phi$ is moreover volume-normalized
Prop \ref{prop:prop arithm vol in toric case} gives

\begin{align}
2\widehat{\text{vol}}_{\chi}\left(\mathcal{-}\mathcal{K},\phi\right)/\mathrm{vol}(-K_{X})= & -\mathcal{\hat{D}}_{\Z}(\phi)=-\mathcal{\hat{D}}_{\psi_{P}}(\phi)\label{eq:Ding Z in toric case}\\
= & -\int_{P}\phi^{*}dy/\mathrm{vol}(-K_{X})+\log\int_{\R^{n}}e^{-\phi(x)}dx+n\log\pi,\nonumber 
\end{align}

where $\mathcal{\hat{D}}_{\Z}(\phi)$ and $\mathcal{\hat{D}}_{\psi_{P}}(\phi)$
are the Ding type functionals defined by formula \ref{eq:def of Ding Z}
and formula \ref{eq:def of Ding on non psh}, respectively, and we
have used formula \ref{eq:integral of exp minus phi in toric}.

We start by recording the following explicit formula for the arithmetic
volume of projective space $\P^{n}$, endowed with a volume normalized
Kähler-Einstein metric (which may be assumed to be the metric induced
by the Fubini-Study metric).
\begin{lem}
\label{lem:arithm vol of proj etc}The following formulas hold for
the metrics $\phi_{KE}$ on the anti-canonical line bundles of $\P_{\C}^{n}$
induced by a volume normalized toric Kähler-Einstein metric:
\[
X=\P_{\C}^{n}\implies2\widehat{\mathrm{vol}}_{\chi}\left(\mathcal{-K},\phi_{KE}\right)=\frac{(n+1)^{n}}{n!}\left((n+1)\sum_{k=1}^{n}k^{-1}-n+\log(\frac{\pi^{n}}{n!})\right)>0
\]
 
\end{lem}

\begin{proof}
First consider the case when $X=\P_{\C}^{n},$ whose canonical integral
model is given by $\mathcal{X}=\P_{\Z}^{n}.$ The canonical model
of the anti-canonical line bundle of $\P_{\C}^{n}$ is given by $\mathcal{O}(1)^{\otimes n+1}\rightarrow\P_{\Z}^{n}.$
As shown in \cite[ §5.4]{g-s} (using the induction formula for the
height; see also \cite[Prop 3.10]{sou}) the height $h_{FS}$ of $\mathcal{O}(1)\rightarrow\P_{\Z}^{n}$
endowed with the Fubini-Study metric $\phi_{FS}$ is given by 
\[
h_{FS}=\frac{1}{2}\sum_{k=1}^{n}\sum_{m=1}^{k}m^{-1}.
\]
Since $(n+1)\phi_{FS}$ defines a Kähler-Einstein metric on $-K_{\P^{n}}$
and $\pi^{-n}\int_{\P^{n}}e^{-(n+1)\phi_{FS}}=1/n!$ this gives
\[
2\widehat{\text{vol}}_{\chi}\left(-\mathcal{K},\phi_{KE}\right)-n\log\pi=(n+1)^{n+1}\frac{h_{FS}}{(n+1)!}+\frac{(n+1)^{n}}{n!}\log(\frac{1}{n!})=\frac{(n+1)^{n}}{n!}\left(h_{FS}+\log(\frac{1}{n!})\right),
\]
 using formula \ref{eq:vol chi as intersection} in the first term,
combined with the homogeneity property \ref{eq:vol chi tensor} and,
in the second term, the scaling property \ref{eq:scaling of chi vol}.
Rewriting the formula for $h_{FS}$ above as a triangle sum and changing
the order of summation then concludes the proof of the formula of
the lemma. The last positivity statement will be shown in the course
of the proof of Lemma \ref{lem:product is maximal implies thm}.
\end{proof}
The key ingredient in the proof of Theorem \ref{thm:main toric intro}
is the following universal bound on the arithmetic volume, in terms
of the ordinary volume:
\begin{prop}
\label{prop:universal toric bound}For any $n-$dimensional toric
Fano variety $X$ which is K-semistable, the following bound holds
for any volume-normalized continuous metric $\phi$ on $-K_{X},$
\[
2\widehat{\mathrm{vol}}_{\chi}\left(\mathcal{-}\mathcal{K},\phi\right)\leq-\mathrm{vol}(X)\log\left(\frac{\mathrm{vol}(X)}{(2\pi^{2})^{n}}\right),\,\,\mathrm{vol}(X):=\mathrm{vol}(-K_{X}).
\]
\end{prop}

\begin{proof}
First recall that, as shown in the beginning of the proof of Theorem
\ref{thm:arithm Vol and K semi st}, it is equivalent to establish
the upper bound for $-\mathcal{\hat{D}}_{\psi_{P}}(\phi)$ when $\phi$
is a continuous psh metric on $L.$ Since $X$ is assumed K-semistable
it follows from Prop \ref{prop:toric ke} that $X$ admits a $T-$invariant
Kähler-Einstein metric. In general, a Kähler-Einstein metric $\phi$
on $-K_{X}$ minimizes the normalized Ding functional $\mathcal{\hat{D}}_{\psi_{0}}$
\cite{bbegz}. Thus in the toric case the infimum of $\mathcal{\hat{D}}_{\psi_{0}}$
coincides with the infimum over all continuous $T-$invariant psh
metrics. As explained in Section \ref{subsec:Logarithmic-coordinates-and}
such a metric may be identified with a convex function $\phi(x)$
on $\R^{n}$ satisfying $\phi-\psi_{P}\in L^{\infty}(\R^{n}).$ By
formula \ref{eq:Ding Z in toric case} it will be enough to show that
for such convex functions
\begin{equation}
-\int_{P}\phi^{*}dy/V+\log\int_{\R^{n}}e^{-\phi(x)}dx\leq-\log V+n\log(2\pi),\,\,\,V:=\text{vol}(-K_{X}).\label{eq:pf of universal}
\end{equation}
 Since $0$ is contained in the interior of $P$ the measure $e^{-\phi}dx$
on $\R^{n}$ has finite moments. Recall that, by Prop \ref{prop:toric ke}
the barycenter of $P$ coincides with $0\in\R^{n}$ and, as a consequence,
the left hand side in inequality \ref{eq:pf of universal} is invariant
under translations of $\phi$, $\phi(x)\mapsto\phi(x+a)$ for any
given $a\in\R^{n}$ \cite[Lemma 2.14]{ber-ber}. As a consequence,
in order to prove the inequality \ref{eq:pf of universal} we may
as well assume that 
\[
\int_{\R^{n}}xe^{-\phi}dx=0.
\]
 By the functional form of Santaló's inequality \cite[Lemma 2.14]{a-k-m}
this implies that 
\[
\int_{\R^{n}}e^{-\phi^{*}(y)}dy\cdot\int_{\mathbb{R}^{n}}e^{-\phi(x)}dx\leq(2\pi)^{n}
\]
(where equality holds if $\phi=\phi^{*}$ i.e. if $\phi(x)=|x|^{2}/2).$
Moreover, by Jensen's inequality 
\[
-\int_{P}\phi^{*}d\lambda/V\leq\log\left(\int_{P}e^{-\phi^{*}(y)}dy/V\right)=\log\left(\int_{\R^{n}}e^{-\phi^{*}(y)}dy/V\right),
\]
 using in the last equality that $\phi^{*}=\infty$ on the complement
of $P$ (see formula \ref{eq:P in terms of Legendre}). Combining
the latter two inequalities yields the desired inequality \ref{eq:pf of universal}.
\end{proof}
Recall that $\P^{n}$ has maximal volume among all K-semistable $n-$dimensional
Fano varieties (as shown in \cite{ber-ber2} in the toric case and
in \cite{fu} in general). We next show that it will be enough to
prove that, in the toric case, the next to largest volume is attained
by $\mathbb{P}^{n-1}\times\mathbb{P}^{1}:$
\begin{lem}
\label{lem:product is maximal implies thm}For any $n-$dimensional
toric Fano variety $X$ which is K-semistable
\[
\mathrm{vol}(X)\leq\mathrm{vol}(\mathbb{P}^{n-1}\times\mathbb{P}^{1})\implies\widehat{\mathrm{vol}}_{\chi}\left(\mathcal{-}\mathcal{K},\phi\right)<\widehat{\mathrm{vol}}_{\chi}\left(-\overline{\mathcal{K}_{\P_{\Z}^{n}}}\right)
\]
where $-K_{\P^{n}}$ is endowed with the volume-normalized Fubini-Study
metric.
\end{lem}

\begin{proof}
First observe that  the function of $\text{vol}(X)$ appearing in
the rhs of the inequality in the previous proposition is increasing
when $\text{vol}(X)\leq(2\pi^{2})^{n}/e.$ This bound is, in fact,
satisfied for any K-semistable $X.$ Indeed, by \cite{ber-ber2},
\begin{equation}
\text{vol}(X)\leq\text{vol}(\mathbb{P}^{n})=\frac{(n+1)^{n}}{n!}<(2\pi^{2})^{n}/e.\label{eq:upper bound on vol in pf}
\end{equation}
(using, in the last inequality a simple induction argument). Thus,
by the previous proposition,  it will be enough to show that

\begin{equation}
-\text{vol}(\mathbb{P}^{n-1}\times\mathbb{P}^{1})\log(\text{vol}(\mathbb{P}^{n-1}\times\mathbb{P}^{1})/(2\pi^{2})^{n})<2\widehat{\text{vol}}_{\chi}\left(-\overline{\mathcal{K}_{\P_{\Z}^{n}}}\right).\label{eq: soughtafter inequality}
\end{equation}
for any $n\geq2.$ To this end first note that
\[
-\text{vol}(\mathbb{P}^{n-1}\times\mathbb{P}^{1})\log(\text{vol}(\mathbb{P}^{n-1}\times\mathbb{P}^{1})/(2\pi^{2})^{n})=-\frac{2n^{n-1}}{(n-1)!}(\log(\frac{2n^{n-1}}{(n-1)!})-n\log(2\pi^{2}))
\]
We check that the inequality holds for $n=2$ and with induction in
mind we simplify the right hand side of \ref{eq: soughtafter inequality}
with $n+1$ for $n$ and get
\begin{align*}
2\widehat{\text{vol}}_{\chi}\left(-\overline{\mathcal{K}_{\P_{\Z}^{n+1}}}\right)= & -\frac{(n+2)^{n+1}}{(n+1)!}(n+1-(n+2)\sum_{k=1}^{n+1}\frac{1}{k}+\log((n+1)!)-(n+1)\log(\pi))\\
= & -(\frac{n+2}{n+1})^{n+1}\frac{(n+1)^{n}}{n!}((n-(n+1)\sum_{k=1}^{n}\frac{1}{k}+\log(n!)-n\log(\pi))+\\
 & (1-(n+2)\sum_{k=1}^{n+1}\frac{1}{k}+(n+1)\sum_{k=1}^{n}\frac{1}{k}+\log(n+1)-\log(\pi)))\\
= & (\frac{n+2}{n+1})^{n+1}2\widehat{\text{vol}}_{\chi}\left(-\overline{\mathcal{K}_{\P_{\Z}^{n}}}\right)-(\frac{n+2}{n+1})^{n+1}(1-\log(\pi)+\log(n+1)-\frac{n+2}{n+1}-\sum_{k=1}^{n}\frac{1}{k})
\end{align*}
Here we observe for later use that $\widehat{\text{vol}}_{\chi}\left(-\overline{\mathcal{K}_{\P_{\Z}^{n}}}\right)>0\forall n\geq1$
by evaluating it at $n=1$ and then using the above to perform induction
and noting that 
\[
-(1-\log(\pi)+\log(n+1)-\frac{n+2}{n+1}-\sum_{k=1}^{n}\frac{1}{k})>-(-\log(\pi)+\log(2))=\log(\frac{\pi}{2})>0
\]
for $n\geq1$. We have used that $-\log(n+1)+\sum_{k=1}^{n}\frac{1}{k}$
is increasing and can thus be estimated from below by putting $n=1$.
We also simplify the left hand side of \ref{eq: soughtafter inequality},

\begin{align*}
-\text{vol}(\mathbb{P}^{n}\times\mathbb{P}^{1})\log(\text{vol}(\mathbb{P}^{n}\times\mathbb{P}^{1})/(2\pi^{2})^{n+1})= & -\frac{2(n+1)^{n}}{n!}(\log(\frac{2(n+1)^{n}}{n!})-(n+1)\log(2\pi^{2}))\\
= & -(\frac{n+1}{n})^{n}\frac{2n^{n}}{n!}((\log(\frac{2n^{n}}{n!})-n\log(2\pi^{2}))+\\
 & (\log((\frac{n+1}{n})^{n}-\log(2\pi^{2}))\\
= & -(\frac{n+1}{n})^{n}\text{vol}(\mathbb{P}^{n-1}\times\mathbb{P}^{1})\log(\text{vol}(\mathbb{P}^{n-1}\times\mathbb{P}^{1})/(2\pi^{2})^{n})\\
 & -2\frac{(n+1)^{n}}{n!}(-\log((\frac{n+1}{n})^{n})-\log(2\pi^{2})).\\
\end{align*}
Fix $n\geq2$ and assume $-\text{vol}(\mathbb{P}^{n-1}\times\mathbb{P}^{1})\log(\text{vol}(\mathbb{P}^{n-1}\times\mathbb{P}^{1})/(2\pi^{2})^{n})\leq2\widehat{\text{vol}}_{\chi}\left(-\overline{\mathcal{K}_{\P_{\Z}^{n}}}\right)$.
Define for brevity $e_{n}=(1+\frac{1}{n})^{n}$ and estimate

\begin{align*}
2\widehat{\text{vol}}_{\chi}\left(-\overline{\mathcal{K}_{\P_{\Z}^{n+1}}}\right) & -(-\text{vol}(\mathbb{P}^{n}\times\mathbb{P}^{1})\log(\text{vol}(\mathbb{P}^{n}\times\mathbb{P}^{1})/(2\pi^{2})^{n+1}))\\
= & e_{n+1}\widehat{\text{vol}}_{\chi}\left(-\overline{\mathcal{K}_{\P_{\Z}^{n}}}\right)-(-e_{n}\text{vol}(\mathbb{P}^{n}\times\mathbb{P}^{1})\log(\text{vol}(\mathbb{P}^{n}\times\mathbb{P}^{1})/(2\pi^{2})^{n}))\\
 & +2\frac{(n+1)^{n}}{n!}(\log(\frac{(n+1)^{n}}{n})-\log(2\pi^{2}))\\
 & -\frac{(n+2)^{n+1}}{(n+1)!}(1-\log(\pi)+\log(n+1)-\frac{n+2}{n+1}-\sum_{k=1}^{n}\frac{1}{k})\\
> & 2\frac{(n+1)^{n}}{n!}(\log(\frac{(n+1)^{n}}{n})-\log(2\pi^{2}))\\
 & -\frac{(n+2)^{n+1}}{(n+1)!}(1-\log(\pi)+\log(n+1)-\frac{n+2}{n+1}-\sum_{k=1}^{n}\frac{1}{k})\\
= & \frac{(n+2)^{n+1}}{(n+1)!}(\frac{(n+1)^{n}}{n!}/\frac{(n+2)^{n+1}}{(n+1)!}2(\log((\frac{n+1}{n})^{n})-\log(2\pi^{2}))\\
 & -1+\log(\pi)-\log(n+1)+\frac{n+2}{n+1}+\sum_{k=1}^{n}\frac{1}{k})\\
= & \frac{(n+2)^{n+1}}{(n+1)!}[\frac{2}{e_{n}}(\log(e_{n})-\log(2\pi^{2}))+\log(\pi)+\sum_{k=1}^{n+1}\frac{1}{k}-\log(n+1)].\\
\end{align*}
In the inequality above we have used $\widehat{\text{vol}}_{\chi}\left(-\overline{\mathcal{K}_{\P_{\Z}^{n}}}\right)>0\forall n\geq1$
and $e_{n}<e_{n+1}$ and the induction hypothesis. Next check numerically
that this last expression is positive for $n=2,3$. For $n\geq4$
we have 
\begin{align*}
\frac{2}{e_{n}}(\log(e_{n})-\log(2\pi^{2}))+\log(\pi) & +\sum_{k=1}^{n+1}\frac{1}{k}-\log(n+1)\\
> & \frac{2}{e_{4}}(\log(e_{4})-\log(2\pi^{2}))+\log(\pi)+\gamma>0.
\end{align*}
We used again that $e_{n}<e_{n+1}$ and the fact that $\sum_{k=1}^{n+1}\frac{1}{k}-\log(n+1)>\gamma$
\cite{ti-ty}, where $\gamma$ is the Euler-Mascheroni constant. The
last inequality is checked numerically.
\end{proof}
We expect that any K-semistable toric Fano variety $X,$ not equal
to $\P^{n},$ satisfies the volume bound in the previous lemma (see
the following section). Here we will show that this is the case under
the conditions of Theorem \ref{thm:main toric intro}. First, the
singular cases are handled using the following bound.
\begin{lem}
\label{lem:sing vol}Let $X$ be a singular K-semistable toric Fano
variety. Then 
\[
\mathrm{vol}(-K_{X})\leq\frac{1}{2}(n+1)^{n}/n!
\]

if anyone of the following conditions hold:
\begin{itemize}
\item $X$ is $\Q-$factorial (or equivalently, $X$ has abelian quotient
singularities).
\item $X$ is not Gorenstein
\end{itemize}
In particular, by the first point, when $n=2$ this inequality holds
for any singular K-semistable toric Fano variety $X.$ 
\end{lem}

\begin{proof}
The result concerning the first point is the toric case of \cite[Thm 3]{liu}
concerning quotient singularities, but in the toric case it also follows
from the proof of \cite[Thm 1.2]{ber-ber2}. For future reference
we recall the argument in \cite{ber-ber2}. Let $P$ be a given polytope
with rational vertices and represent $P$ as the intersection of hyperplanes
$\left\{ p\in\R^{n}:\,\left\langle l_{F},p\right\rangle \geq-a_{F}\right\} ,$
where the index $F$ ranges over the facets of $P,$ $l_{F}$ is a
primitive vector in $\Z^{n}$ and $a_{F}$ is a non-zero positive
numbers. In the present Fano case $a_{F}=1.$ Moreover, since $X$
is assumed to be $\Q-$factorial for any vertex $p_{0}$ of $P$ there
are precisely $n$ facets $F_{1},...,F_{n}$ of $P$ intersecting
$p_{0},$ numbered so that the corresponding normals define a positively
oriented bases in $\R^{n}$ \cite{c-l-s}. Fixing a vertex $p_{0}$
of $P$ we denote by $P'$ the image of $P$ under the map 
\begin{equation}
p\mapsto\left(\frac{\left\langle l_{F_{1}},p\right\rangle +a_{F_{1}}}{a_{F_{1}}},...,\frac{\left\langle l_{F_{n}},p\right\rangle +a_{F_{n}}}{a_{F_{n}}}\right),\label{eq:map in pf sing vol}
\end{equation}
 which is a polytope in $[0,\infty[^{n}.$ Moreover, assuming that
$0$ is the barycenter of $P$ the barycenter of $P'$ is $(1,...,1).$
By \cite[Thm 1.5]{ber-ber2} the volume $\mathrm{Vol}(P')$ of any
such polytope is maximal when $P'$ is $(n+1)$ times the unit-simplex
in $[0,\infty[^{n}$ with vertex at $(0,...,0).$ Hence, 
\begin{equation}
\mathrm{Vol}(P')\leq(n+1)^{n}/n!,\ \ \ \mathrm{Vol}(P')=\frac{\delta}{a_{F_{1}}\cdots a_{F_{n}}}\mathrm{Vol}(P)\label{eq:ineq in pf lemma vol sing}
\end{equation}

where $\delta$ is the determinant of the map $p\mapsto\left(\left\langle l_{F_{1}},p\right\rangle ,...,\left\langle l_{F_{n}},p\right\rangle \right).$
Thus $\delta$ is a positive integer and $\delta=1$ iff the map is
invertible, i.e. if and only if $l_{F_{1}},...,l_{F_{n}}$ generate
$\Z^{n},$ which is equivalent to the $T_{c}-$invariant \foreignlanguage{american}{neighbourhood}
$U_{0}$ corresponding to the vertex $p_{0}$ being biholomorphic
to $\C^{n}$ \cite{c-l-s}. Hence, if $X$ is singular (i.e. $X$
is not non-singular), then there must be some vertex $p_{0}$ with
$\delta\geq2.$ Since $a_{F_{i}}=1$ this concludes the proof.

To prove the second point we employ a similar argument. This time,
for $X$ possibly not $\Q$-factorial, there might be more than $n$
facets intersecting a vertex $p_{0}.$ Still, there are at least $n$
facets intersecting at $p_{0}$, and we can construct the map \ref{eq:map in pf sing vol}
by choosing any $n$ of them. Next note that if $\delta=1$, the map
and its inverse have integer coefficients (since $a_{F_{i}}=1$ when
$X$ is Fano) and since $p_{0}$ is mapped to $0$, $p_{0}\in\mathbb{Z}^{n}$.
Since $p_{0}$ was arbitrary, it follows that $P$ is a lattice polytope
and hence $X$ is Gorenstein. Thus $\delta\geq2$ and we are done. 
\end{proof}
The volume bound in the previous lemma implies the volume bound in
Lemma \ref{lem:product is maximal implies thm} is satisfied:
\begin{equation}
\frac{(n+1)^{n}}{2n!}\leq\frac{2n^{n-1}}{(n-1)!}\iff\left(1+1/n\right)^{n}\leq4.\label{eq:halv Pn vol smaller than}
\end{equation}
 The lhs in the latter inequality increases to $e,$ which is, indeed,
smaller than $4.$ This proves Theorem \ref{thm:main toric intro}
in the singular cases. Finally, in the case that $X$ is non-singular
there are, for any given dimension $n$ only a finite number of cases
to check in order to verify  the volume bound in Lemma \ref{lem:product is maximal implies thm}.
When $n\leq6$ we may apply the database \cite{ob} of all non-singular
Fano varieties of dimension $n.$ The condition that the barycenter
of $P$ vanishes, corresponds in the data base to the condition ``zero
dual barycentre''. Adding the condition $(-K_{X})^{n}\geq n!\text{vol}(\mathbb{P}^{n-1}\times\mathbb{P}^{1})$
the database only furnishes $\P^{n}$ and $\mathbb{P}^{n-1}\times\mathbb{P}^{1},$
as desired.

\subsubsection{\label{subsec:gap}Remarks on the ``gap hypothesis''}

In order to extend the proof of Theorem \ref{thm:main toric intro}
to a any dimension $n$ one would need to establish the following
conjecture (established above under the conditions in Theorem \ref{thm:main toric intro}):
\begin{conjecture}
\label{conj:(the-gap-hypothesis).}(the ``gap hypothesis''). For
any $n-$dimensional toric K-semistable Fano manifold $X$ different
from $\P^{n},$ $\mathrm{vol}(X)\leq\mathrm{vol}(\mathbb{P}^{n-1}\times\mathbb{P}^{1}).$
\end{conjecture}

This conjecture might even hold without the toric assumptions in any
dimension (as pointed out to us by Ziquan Zhuang this appears to be
a folklore conjecture). For example, when $n=3$ and $X$ is non-singular
it follows from the well-known classification of three dimensional
Fano manifolds (see the ``big table'' in \cite[Section 6]{ar})
that the only Fano manifolds $X,$ different from $\P^{3},$ which
do not satisfy the inequality in question are $\P^{3}$ blown-up in
one point and $\P(\mathcal{O}\oplus\mathcal{O}(2)).$ But both of
these are K-unstable, i.e. they are not K-semistable. Indeed, these
two Fano manifolds are toric and if they were K-semistable they would
satisfy the gap hypothesis, by the toric case $(n\leq6)$ applied
to $n=3.$ Let us also point out that in the toric case it is only
$\P^{n-1}\times\P^{1}$ that saturates the inequality in the ``gap
hypothesis'' when $n\leq6$ and it seems thus natural to ask if this
is also the case when $n>6?$ However, in the general non-toric case
the inequality is also saturated by the non-singular quadratic hypersurface
$X_{2}$ in $\P^{n+1},$ i.e. the base of the Ordinary Double Point
 (ODP). Moreover, as pointed out to us by Yuji Odaka, in the general
case our ``gap hypothesis'' is reminiscent of the ODP-conjecture
in \cite{s-s}, very recently settled in the toric case \cite{m-s}.
More precisely, in our setup, the ODP-conjecture implies that 
\begin{equation}
\mathrm{vol}(X)\leq\mathrm{vol}(\mathbb{P}^{n-1}\times\mathbb{P}^{1})(n/I(X)),\label{eq:ODP}
\end{equation}
 where $I(X)$ denotes the Fano index of $X$ (i.e. largest positive
integer such that $K_{X}/I(X)$ is a line bundle). However, $I(X)\leq n$
when $X\neq\P^{n}$ (with equality iff $X=X_{2})$ and hence the inequality
\ref{eq:ODP} is weaker than our ``gap hypothesis''.

\subsubsection{\label{subsec:The-case-of-toric products}The case of products in
any dimension}
\begin{lem}
The ``gap hypothesis'' holds when $X$ is the product of K-semistable
Fano varieties $X_{1},....,X_{M}$ (not necesseraily assumed toric),
for $M\geq2$. 
\end{lem}

\begin{proof}
By a simple induction argument we may as well assume that $M=2.$
Let, without loss of generality, $n:=\text{dim}(X_{1})\geq\dim(X_{2})=:m>1$.
Note that if $m=1$ we are done since then, $\text{vol}(X)=\text{vol}(X_{1})\text{vol}(X_{2})\leq\text{vol}(\mathbb{\mathbb{P}}^{N-1})\text{vol}(\mathbb{P}^{1})=\text{vol}(\mathbb{P}^{N-1}\times\mathbb{P}^{1})$
using that, by Fujita's inequality \ref{eq:fujita intro}, complex
projective space maximizes the volume among K-semistable Fano varieties
in each dimension. Using again that complex projective space maximizes
the volume in each given dimension and defining for brevity $e_{k}:=(1+\frac{1}{k})^{k}$
we get $\text{vol}(X)=\text{vol}(X_{1})\text{vol}(X_{2})\leq\text{vol}(\mathbb{\mathbb{P}}^{n})\text{vol}(\mathbb{P}^{m})$
\[
=\frac{(n+1)^{n}}{n!}\frac{(m+1)^{m}}{m!}=\frac{(n+2)^{n+1}}{(n+1)!}\frac{m^{m-1}}{(m-1)!}(\frac{n+1}{n+2})^{n+1}(\frac{m+1}{m})^{m}
\]
\[
=\frac{(n+2)^{n+1}}{(n+1)!}\frac{m^{m-1}}{(m-1)!}\frac{e_{m}}{e_{n+1}}<\frac{(n+2)^{n+1}}{(n+1)!}\frac{m^{m-1}}{(m-1)!}=\text{vol}(\mathbb{\mathbb{P}}^{n+1})\text{vol}(\mathbb{P}^{m-1})
\]
where in the last inequality we have used that $e_{k}$ is increasing
in $k$. We may continue in similar manner until we have $\text{vol}(\mathbb{P}^{N-1}\times\mathbb{P}^{1})$
in the right hand side and we are done. 
\end{proof}
As explained in the previous section, it follows from the previous
lemma that Conjecture \ref{conj:height intro} holds when $\mathcal{X}$
is a product of toric arithmetic Fano varieties, i.e. $\mathcal{X}=\mathcal{X}_{1}\times\cdots\times\mathcal{X}_{M},$
where $\mathcal{X}_{i}$ is endowed with its canonical integral structure. 

\subsection{The height of toric Kähler-Einstein metrics; proof of Theorem \ref{thm:ke intro}}

By Prop \ref{prop:universal toric bound} it only remains to prove
the lower bound. Using the notation in the proof of Prop \ref{prop:universal toric bound}
we have that, for any continuous convex function $\psi$ on $\R^{n}$
such that $\psi-\psi_{P}$ is bounded, 
\[
2(\overline{-\mathcal{K}_{\mathcal{X}}})^{n+1}/\mathrm{vol}(-K_{X})\geq-\int_{P}\psi^{*}dy/\mathrm{Vol}(P)+\log\int_{\R^{n}}e^{-\psi}dx+n\log\pi
\]
 In particular, taking $\psi=\psi_{P}$ the first term in the right
hand side vanishes. Moreover, 
\[
I:=\int_{\R^{n}}e^{-\psi_{P}}dx=n!\mathrm{Vol}\text{\ensuremath{(P^{*})}},
\]
 where $P^{*}$ denotes the polar dual of $P,$ i.e. $P^{*}$ consists
of all $x\in\R^{n}$ such that $x\cdot p\leq1$ for all $p\in P.$
Indeed, 
\[
I=\int_{[0,\infty[}e^{-t}(\psi_{P})_{*}dx=\int e^{-t}\frac{dV(t)}{dt}dt=\int e^{-t}V(t)dt=\int_{0}^{\infty}e^{-t}t^{n}dt\mathrm{Vol}\ensuremath{(P^{*})},
\]
 where $V(t)$ is the Lebesgue volume of $\{\psi_{P}<t\}$ i.e. of
$tP^{*}.$ Hence, 
\[
2(\overline{-\mathcal{K}_{\mathcal{X}}})^{n+1}\geq\mathrm{Vol}(P)\left(\log\left(n!\mathrm{Vol}(P^{*})\right)+n\log\pi\right).
\]
 Since, by definition, $\mathrm{Vol}(P^{*})\mathrm{Vol}(P)\geq m_{n}$
this concludes the proof of the lower bound in the theorem. Next,
by \cite[Cor 1.8]{ku} (see also \cite{bern})
\[
m_{n}\geq(\frac{\pi}{2e})^{n-1}(n+1)^{n+1}/(n!)^{2}=(\frac{\pi}{2e})^{n-1}\frac{(n+1)}{n!}\sigma_{n},
\]
 where $\sigma_{n}=\text{\ensuremath{\mathrm{vol}}\ensuremath{(\P^{n}).}}$
Since $\mathrm{Vol}(P)\leq\sigma_{n}$ (by \ref{eq:upper bound on vol in pf})
this means that

\[
n!\pi^{n}m_{n}\mathrm{Vol}(P)^{-1}\geq n!\pi^{n}m_{n}\sigma_{n}^{-1}=\pi(\frac{\pi^{2}}{2e})^{n-1}(n+1)>1
\]
 proving the positivity in the theorem.

\subsection{Examples}

We next provide examples of families of toric varities $X$ for which
the height of the corresponding Kähler-Einstein can be explicitely
computed as a function of $\mathrm{vol}(X)$ of the same form as in
Theorem \ref{thm:ke intro}. The examples are based on the following 
\begin{prop}
\label{prop: linearly equivalent polytopes}Let $X_{1}$ and $X_{2}$
be two K-semistable toric Fano varieties of dimension $n$ with moment
polytopes $P_{1}$ and $P_{2}$ such that $P_{2}=AP_{1}$ for an invertible
linear transformation $A$ (the polytopes are linearly equivalent).
Denote the canonical integral models of $X_{1}$ and $X_{2}$ by $\mathcal{X}_{1}$
and $\mathcal{X}_{2}$ respectively. Then, with heights taken with
respect to the volume-normalized Kähler-Einstein metrics, 

\[
\frac{(\overline{-\mathcal{K}_{\mathcal{X}_{2}}})^{n+1}/(n+1)!}{\left(-K_{X_{2}}\right)^{n}/n!}=\frac{(\overline{-\mathcal{K}_{\mathcal{X}_{1}}})^{n+1}/(n+1)!}{\left(-K_{X_{1}}\right)^{n}/n!}-\frac{1}{2}\log\mathrm{det}A.
\]
\end{prop}

As a consequence, for $X$ a K-semistable toric Fano variety of dimension
$n$,
\begin{equation}
(\overline{-\mathcal{K}_{\mathcal{X}}})^{n+1}=\frac{(n+1)!}{2}\mathrm{vol}(X)\log\left(\frac{a}{\mathrm{vol}(X)}\right)\label{eq:family formula}
\end{equation}

where $a$ is a constant independent of the choice of $X$ within
a class of toric varieties with linearly equivalent moment polytopes.
More precisely, 
\begin{equation}
a=\mathrm{vol}(X)\exp(\frac{2(\overline{-\mathcal{K}_{\mathcal{X}}})^{n+1}/(n+1)!}{\mathrm{vol}(X)})\label{eq:family formula constant}
\end{equation}
and Proposition \ref{prop: linearly equivalent polytopes} ensures
the claimed independence. 
\begin{proof}
(of Proposition \ref{prop: linearly equivalent polytopes}) Recall
that, with heights taken with respect to Kähler-Einstein metrics,
\[
\frac{(\overline{-\mathcal{K}_{\mathcal{X}_{2}}})^{n+1}/(n+1)!}{\left(-K_{X_{2}}\right)^{n}/n!}=-\frac{1}{2}\sup_{\phi}-\frac{1}{\mathrm{vol}(P_{2})}\int_{P_{2}}\phi^{*}(p)\mathrm{d}p+\log\int_{\mathbb{R}^{n}}\exp(-\phi(x))\mathrm{d}x.
\]
Changing variables in the integrals, $p\mapsto A^{t}p'$ and $x\mapsto Ax'$
we get
\[
\frac{(\overline{-\mathcal{K}_{\mathcal{X}_{2}}})^{n+1}/(n+1)!}{\left(-K_{X_{2}}\right)^{n}/n!}=-\frac{1}{2}(\sup_{\phi(A\cdot)}-\frac{1}{\mathrm{vol}(P_{1})}\int_{P_{1}}\phi^{*}(A^{t}p')\mathrm{d}p'+\log\int_{\mathbb{R}^{n}}\exp(-\phi(Ax'))\mathrm{d}x'+\mathrm{\log detA)}.
\]
Next we rename $\phi'=\phi(A\cdot)$ and use that then $\phi'^{*}=\phi^{*}(A^{t}\cdot)$
to get the result. 
\end{proof}
\begin{example}
Recall the K-semistable toric Fano varieties $X_{q,p}$ parametrized
with two prime numbers from Example \ref{exa:toric family}. The corresponding
polytope $P(-K_{X_{p,q}})$ is the image of the polytope $P(-K_{\mathbb{P}^{1}\times\mathbb{P}^{1}})=\mathrm{conv}\{(1,1),(1,-1),(-1,1),(-1,-1)\}$
under the linear map $A$ given in matrix form by $\begin{bmatrix}\frac{1}{2p} & \frac{1}{2p}\\
\frac{-1}{2q} & \frac{1}{2q}
\end{bmatrix}$. Thus the family $\mathcal{F}=\{\mathbb{P}^{1}\times\mathbb{P}^{1},X_{p,q}:p,q\ \mathrm{prime\}}$
comprise an example of a family of K-semistable toric Fano varieties
with linearly equivalent moment polytopes. Thus by \ref{eq:family formula},
for $X\in\mathcal{F}$,
\[
(\overline{-\mathcal{K}_{\mathcal{X}}})^{n+1}=\frac{(n+1)!}{2}\mathrm{vol}(X)\log\left(\frac{a}{\mathrm{vol}(X)}\right)
\]
with, by \ref{eq:family formula constant}, \ref{lem:arithm vol of proj etc}
and a simple computation,
\[
a=\mathrm{vol}(\mathbb{P}^{1}\times\mathbb{P}^{1})\exp(\frac{2(\overline{-\mathcal{K}_{\mathbb{P}^{1}\times\mathbb{P}^{1}}})^{n+1}/(n+1)!}{\mathrm{vol}(\mathbb{P}^{1}\times\mathbb{P}^{1})})=4\exp(2-\log\pi^{2}).
\]
Recall also that $\text{vol \ensuremath{(-K_{X_{p,q}})}}=2/(pq)$
so that in this family the heights with respect to the Kähler-Einstein
metrics are explicitely computed by the previous formula.
\end{example}

\section{\label{sec:Donaldson's-toric-invariant}Sharp bounds on Donaldson's
toric Mabuchi functional }

Let $(X,L)$ be a polarized complex manifold and denote by $\mathcal{H}(X,L)$
the space of all smooth metrics $\psi$ on $L$ whose curvature form
$dd^{c}\psi$ is positive, $dd^{c}\psi>0.$

\subsection{\label{subsec:The-Mabuchi-functional}The Mabuchi functional (recap)}

The Mabuchi functional $\mathcal{M}$ on $\mathcal{H}(X,L)$ is defined,
up to addition by a constant, by declaring that its differential on
$\mathcal{H}(X,L)$ at a given point $\psi$ is represented by the
following measure on $X:$ 
\begin{equation}
d\mathcal{M}_{|\psi}:=\left(-S(\psi)+a\right)\frac{(dd^{c}\psi)^{n}}{n!},\,\,\,a:=n(-K_{X})\cdot L^{n-1}/L^{n},\label{eq:def of dM}
\end{equation}
 where $S(\psi)$ denotes the scalar curvature of the Kähler form
$(dd^{c}\psi),$ i.e. the trace of the Ricci curvature:
\[
S(\psi)\frac{(dd^{c}\psi)^{n}}{n!}:=\text{Ric }(dd^{c}\psi)\wedge\frac{(dd^{c}\psi)^{n-1}}{(n-1)!}.
\]
 Recall that the Ricci curvature $\text{Ric}(dd^{c}\psi)$ of the
Kähler form $dd^{c}\psi$ is the $(1,1)-$form defined as the curvature
of the metric on $-K_{X}$ induced by the volume form of $dd^{c}\psi.$
We have followed Donaldson's multiplicative normalizations in \cite[formula 3.2.1]{do1},
which differ from the original definition in \cite{mab}, where the
measure $\frac{(dd^{c}\psi)^{n}}{n!}$ on $X$ is volume-normalized.
At any rate, formula \ref{eq:def of dM} only determines the Mabuchi
functional $\mathcal{M}$ up to an additive constant. 

\subsubsection{The case when $X$ is a Fano manifold and $L=-K_{X}$}

We now specialize to the case when $L=-K_{X}$ and note that a choice
of reference metric $\psi_{0}$ in $\mathcal{C}^{0}(L)\cap\text{PSH\ensuremath{(L)}}$
induces a particular choice of Mabuchi functional, i.e. a functional
whose differential satisfies formula \ref{eq:def of dM}, that we
shall denote by $\mathcal{M}_{\psi_{0}}.$ This is a consequence of
the thermodynamical formalism introduced in \cite{berm6}, which expresses
\begin{equation}
\mathcal{M}_{\psi_{0}}(\psi):=\text{vol\ensuremath{(-K_{X})}}F_{\psi_{0}}\left(MA(\psi)\right),\label{eq:def of M psi not}
\end{equation}
 where $MA(\psi)$ is the probability measure on $X$ defined by the
normalized volume form of the Kähler metric $dd^{c}\psi:$
\begin{equation}
MA(\psi):=\frac{1}{n!}(dd^{c}\psi)^{n}/\text{\ensuremath{\text{vol}(L)}}\label{eq:def of MA}
\end{equation}
and $F_{\psi_{0}}\left(\mu\right)$ denotes the \emph{free energy}
\emph{functional }on the space $\mathcal{P}(X)$ of all probability
measures on $X,$defined as follows:
\begin{equation}
F_{\psi_{0}}(\mu):=-E_{\psi_{0}}(\mu)+\text{Ent}_{dV_{0}}(\mu)\in]-\infty,\infty]\label{eq:def of fri F}
\end{equation}
Here $\text{Ent}_{dV_{0}}(\mu)$ denotes the \emph{entropy} of $\mu$
relative to the volume form $dV_{0}$ on $X$ induced by $\psi_{0}$
(i.e. $dV_{0}=e^{-\psi_{0}}$ in the notation of Section \ref{subsec:Metrics-on- minus KX vs volume})
defined by 

\[
\text{Ent}_{dV_{0}}(\mu):=\int\log\frac{\mu}{dV_{0}}\mu
\]
when $\mu\in L^{1}(X,dV_{0})$ and otherwise $\text{Ent}_{dV_{0}}(\mu):=\infty.$
Furthermore, $E_{\psi_{0}}(\mu)$ is the \emph{pluricomplex energy}
of $\mu,$ relative to $\psi_{0},$ introduced in \cite{bbgz}, which
may be defined as a Legendre-Fenchel transform of the functional $\mathcal{E}_{\psi_{0}}/\text{vol}(L)$
(defined by formula \ref{eq:def of E beauti}). For our purposes it
will be enough to define $E_{\psi_{0}}(\mu)$ when $\mu$ is of the
form $\mu=MA(\psi)$ for $\psi$ in $\mathcal{C}^{0}(L)\cap\text{PSH\ensuremath{(L)}}:$
\begin{equation}
E_{\psi_{0}}(MA(\psi))=\frac{\mathcal{E}_{\psi_{0}}(\psi)}{\text{vol}(L)}-\int_{X}(\psi-\psi_{0})MA(\psi).\label{eq:energy of MA}
\end{equation}
We recall that formula \ref{eq:def of M psi not} follows readily
from the fact that on the subspace of all volume forms $\mu$ in $\mathcal{P}(X)$
the differential of $E_{\psi_{0}}$ at $\mu\in\mathcal{P}(X)$ is
represented by the function $\psi_{0}-\psi_{\mu}:$
\[
dE_{\psi_{0}|\mu}=-(\psi_{\mu}-\psi_{0})
\]
 (this formula is dual to formula \ref{eq:differential of E beauti}
in the sense of Legendre transforms; see \cite{berm6}). 
\begin{rem}
Formula \ref{eq:def of M psi not} defines $\mathcal{M}_{\psi_{0}}(\psi)$
on the space $\mathcal{C}^{0}(L)\cap\text{PSH\ensuremath{(L)}}$ as
a function taking values in $]-\infty,\infty].$ More generally, the
functional $\mathcal{M}_{\psi_{0}}(\psi)$ is well-defined as soon
as $E(MA(\psi))<\infty$ (see \cite{berm6,bbegz}). For $\psi$ smooth
formula \ref{eq:def of M psi not} is essentially equivalent to a
formula for the Mabuchi functional appearing in \cite{ti0} and \cite{ch0}.
\end{rem}

\subsubsection{The case when $X$ is a singular Fano variety}

In the case when $X$ is a singular Fano variety we will denote by
$\mathcal{H}(X,-K_{X})$ the space of all continuous metrics $\psi$
on $L$ such that $\psi$ is smooth on the regular locus $X_{\text{reg}}$
of $X$ and $dd^{c}\psi>0$ on $X_{\text{reg}}.$ 

\subsection{Proof of Theorem \ref{thm:Don inv intro}}

First recall the following basic inequality that holds on any Fano
variety \cite[Lemma 4.4]{bbegz}:

\begin{equation}
F_{\psi_{0}}\left(MA(\psi)\right)\geq\mathcal{\hat{D}}_{\psi_{0}}(\psi)\label{eq:F geq D}
\end{equation}
 as follows from the non-negativity of the relative entropy between
two probability measures (or Jensen's inequality). In fact, the following
identity holds \cite[Lemma 4.4]{bbegz}: 
\begin{equation}
\inf_{\mathcal{C}^{0}(L)\cap\text{PSH\ensuremath{(L)}}}F_{\psi_{0}}\left(MA(\psi)\right)=\inf_{\mathcal{C}^{0}(L)\cap\text{PSH\ensuremath{(L)}}}\mathcal{\hat{D}}_{\psi_{0}}(\psi),\label{eq:inf is inf given reference}
\end{equation}
(the two infima above may, equivalently, be restricted to $\mathcal{H}(X,L);$
see the regularization result in \cite{bdl1}).

Combining Theorem \ref{thm:main toric intro} with the inequality
\ref{eq:F geq D} the proof is concluded by invoking the following
formula relating $\mathcal{M}_{\psi_{P}}$(where $\psi_{P}$ is the
canonical toric reference defined by formula \ref{eq:def of psi P x})
to Donaldson's toric Mabuchi functional 
\begin{equation}
\mathcal{M}_{-K_{X}}(\psi):=\int_{\partial P}\psi^{*}d\sigma-n\int_{P}\psi^{*}dx-\int_{P}\log\det(\nabla^{2}\psi^{*})dx,\label{eq:Ds mab func text}
\end{equation}
 where $\psi^{*}$ denotes the Legendre transform of the $T-$invariant
metric $\psi\in\mathcal{H}(X,-K_{X})$ and $d\sigma$ is the measure
on $\partial P,$ absolutely continuous wrt the $(n-1)-$dimensional
Lebesgue measure $d\lambda_{\partial P},$ defined by $d\sigma=d\lambda_{\partial P}/\left\Vert l_{F}\right\Vert $
on a facet $F$ of $\partial P,$ where $\left\Vert l_{F}\right\Vert $
denotes the Euclidean norm of a primitive normal vector to $F.$
\begin{lem}
\label{lem:M psi P is Donaldson plus log}Let $X$ be an $n-$dimensional
toric Fano variety. The following identity holds on the space of all
$T-$invariant metrics in $\mathcal{H}(X,-K_{X}):$

\[
\mathcal{M}_{\psi_{P}}=\mathcal{M}_{-K_{X}}-\mathrm{vol}(-K_{X})\log\mathrm{vol}(-K_{X})
\]
 
\end{lem}

\begin{proof}
This formula is essentially the content of \cite[Prop 4.6]{ber-ber},
but since the normalizations are a bit different we recall the proof.
First identifying a toric metric $\psi$ with a convex function on
$\R^{n}$ (as in Section \ref{subsec:Logarithmic-coordinates-and})
formula \ref{eq:def of M psi not}, combined with formula \ref{eq:energy of MA},
yields
\[
\mathcal{M}_{\psi_{P}}(\psi)=-\mathcal{E}_{\psi_{P}}(\psi)+\int_{\R^{n}}(\psi-\psi_{P})(dd^{c}\psi)^{n}/n!+\int_{\R^{n}}\log\left(\frac{MA(\psi)}{e^{-\psi_{P}}dx}\right)\mathrm{vol}(-K_{X})MA(\psi)=
\]
\[
=\int_{P}\psi^{*}d\lambda+\int_{\R^{n}}\psi(dd^{c}\psi)^{n}/n!+\int_{\R^{n}}\log\det(\nabla^{2}\psi)\det(\nabla^{2}\psi)-\mathrm{vol}(-K_{X})\log\mathrm{vol}(-K_{X}).
\]
 By \cite[Lemma 4.7]{ber-ber} making the change of variables $y=\nabla\psi$
the second term above may be expressed 
\begin{equation}
\int_{\R^{n}}\psi(dd^{c}\psi)^{n}/n!=\int_{\partial P}\psi^{*}d\sigma-(n+1)\int udp,\label{eq:integral psi ddpsi on Rn}
\end{equation}
 giving 
\[
\mathcal{M}_{\psi_{P}}(\psi)=\int_{\partial P}\psi^{*}d\sigma-n\int_{P}\psi^{*}d\lambda+\int_{\R^{n}}\log\det(\nabla^{2}\psi)\det(\nabla^{2}\psi)-\mathrm{vol}(-K_{X})\log\mathrm{vol}(-K_{X}).
\]
Again making the change of variables $y=\nabla\psi$ in the remaining
integral over $\R^{n}$ concludes the proof, using the standard relation
$\det(\nabla^{2}\psi)(x)\det(\nabla^{2}\psi^{*})(\nabla\psi(x))=1$
(which follows from the fact that the map $y\mapsto\nabla\psi^{*}(y)$
is the inverse of $x\mapsto\nabla\psi(x)$).
\end{proof}

\section{\label{sec:Connections-to-the ar}Relations to the arithmetic Mabuchi
functional}

Given an integral model $(\mathcal{X},\mathcal{L})$ of a polarized
variety $(X,L)$ consider the \emph{arithmetic Mabuchi functional}
$\mathcal{M}_{(\mathcal{X},\mathcal{L})}$ on $\mathcal{H}(X,L)$
defined by

\begin{equation}
\mathcal{M}_{(\mathcal{X},\mathcal{L})}(\psi):=\frac{a}{(n+1)!}\overline{\mathcal{L}}^{n+1}+\frac{1}{n!}\overline{\mathcal{K}}_{\mathcal{X}}\cdot\overline{\mathcal{L}}^{n},\,\,\,\,a=-n(K_{X}\cdot L^{n-1})/L^{n}\label{eq:def of arithm Mab}
\end{equation}
where $\overline{\mathcal{L}}=(\mathcal{L},\psi)$ and $\overline{\mathcal{K}}_{\mathcal{X}}$
is endowed with the metric induced by the measure $MA(\psi)$ on $X,$
i.e. the normalized volume form of the Kähler form $dd^{c}\psi.$
As discussed in Section \ref{subsec:The-arithmetic-K-energy} this
functional coincides, up to additive and multiplicative normalizations,
with the arithmetic Mabuchi functional introduced in \cite{o}.
\begin{lem}
The differential of the functional $\psi\mapsto2\mathcal{M}_{(\mathcal{X},\mathcal{L})}(\mathcal{L},\psi)$
on $\mathcal{H}(X,L)$ satisfies the defining formula \ref{eq:def of dM}
of the Mabuchi functional.
\end{lem}

\begin{proof}
As pointed out in \cite{o} this formula can be deduced from the formula
for the Mabuchi functional in \cite{ti0,ch0}. But for completeness
and to check the normalizations we provide a simple direct proof.
First recall the following property of arithmetic intersection numbers
which holds if $\mathcal{L}_{0}\rightarrow\mathcal{X}$ is the trivial
line bundle (which is a consequence of the restriction formula \cite[Prop 2.3.1]{b-g-s}
and Lemma \ref{lem:Stein}): 
\begin{equation}
\left(\mathcal{L}_{0},\phi_{0}\right)\cdot\left(\mathcal{L}_{1},\phi_{0}\right)\cdot...\cdot\left(\mathcal{L}_{n},\phi_{n}\right)=\frac{1}{2}\int_{X}\phi_{0}dd^{c}\phi_{1}\wedge\cdots\wedge dd^{c}\phi_{n},\label{eq:arithm inters form for trivial}
\end{equation}
 where $\phi_{0}$ is the globally well-defined function on $X$ defined
by formula \ref{eq:def of phi U} when $e_{U}$ is the standard global
trivialization $1$ of the trivial line bundle over $X,$ i.e. $\phi_{0}/2=-\log\left\Vert s\right\Vert _{\phi_{0}},$
where $s$ is a global trivialization of $\mathcal{L}.$ In particular,
differentiating along a curve $t\mapsto\psi_{t}$ in $\mathcal{H}(X,L)$
and using the symmetry of arithmetic intersection numbers gives
\[
\frac{d}{dt}\left(\left(\mathcal{L},\psi_{t}\right)^{n+1}\right)=(n+1)\left(\mathcal{L}_{0},\frac{d\psi_{t}}{dt}\right)\cdot\left(\mathcal{L},\psi_{t}\right)^{n}=\frac{1}{2}\int_{X}\frac{d\psi_{t}}{dt}(dd^{c}\psi)^{n}
\]
 where $\frac{d\psi_{t}}{dt}$ is a globally well-defined function
on $X$ and can thus be identified with a metric on the trivial line
bundle that we denote by $\mathcal{L}_{0}.$ Likewise, denoting by
$\rho_{t}$ a local density for $MA(\psi_{t})$ with respect the Euclidean
measure defined by local holomorphic coordinates,
\begin{equation}
\frac{d}{dt}\left(\left(\mathcal{K}_{\mathcal{X}},\log\rho_{t}\right)\left(\mathcal{L},\psi_{t}\right)^{n}\right)=\left(\mathcal{K}_{\mathcal{X}},\log\rho_{t}\right)n\left(\mathcal{L},\frac{d\psi_{t}}{dt}\right)\cdot\left(\mathcal{L},\psi_{t}\right)^{n-1}+\left(\left(\mathcal{L}_{0},\frac{d}{dt}\log\rho_{t}\right)\cdot\left(\mathcal{L},\psi_{t}\right)^{n}\right)\label{eq:pf of lemma d of arithm mab}
\end{equation}
where we have used Leibniz rule. Applying formula \ref{eq:arithm inters form for trivial},
the second term above may, after multiplication by $2,$ be expressed
as
\[
=\int_{X}\frac{d}{dt}\log\rho_{t}(dd^{c}\psi_{t})^{n}=n!\mathrm{vol}(L)\int_{X}\frac{d}{dt}\log\rho_{t}\rho_{t}=n!\mathrm{vol}(L)\frac{d}{dt}\int_{X}\rho_{t}=0,
\]
using in the last equality that $\int_{X}\rho_{t}=\mathrm{vol}(L)$
for any $t.$ Likewise, applying formula \ref{eq:arithm inters form for trivial}
to the first term in formula \ref{eq:pf of lemma d of arithm mab}
yields 
\[
2\left(\mathcal{K}_{\mathcal{X}},\log\rho_{t}\right)\left(\mathcal{L},\frac{d\psi_{t}}{dt}\right)/n=\int_{X}\frac{d\psi_{t}}{dt}dd^{c}\left(\log\rho_{t}\right)\wedge(dd\psi_{t})^{n-1}=-\int_{X}\frac{d\psi_{t}}{dt}\text{Ric}(dd^{c}\psi_{t})\wedge(dd\psi_{t})^{n-1}.
\]
 All in all, this concludes the proof.
\end{proof}
The following proposition relates the arithmetic Mabuchi functional
$\mathcal{M}_{(\mathcal{X},\mathcal{-K_{\mathcal{X}}})}$ to Donaldson's
toric Mabuchi functional $\mathcal{M}_{-K_{X}}$ (formula \ref{eq:Ds mab func text}):
\begin{prop}
\label{prop:arithm Mab as Don Mab}Given a toric Fano variety $X$
denote by $\mathcal{X}$ its canonical integral model. Then the following
formula holds for any $T-$invariant metric in $\mathcal{H}(X,-K_{X}):$
\[
2\mathcal{M}_{(\mathcal{X},\mathcal{-K_{\mathcal{X}}})}=\mathcal{M}_{-K_{X}}-\mathrm{vol}(-K_{X})\log\mathrm{vol}(-K_{X})
\]
\end{prop}

\begin{proof}
In this case $a=n$ and we can thus decompose $\mathcal{M}_{(\mathcal{X},\mathcal{L})}(\psi)$
as
\begin{equation}
\frac{1}{(n+1)!}\overline{\mathcal{L}}^{n+1}+\frac{1}{n!}(\overline{\mathcal{L}}+\overline{\mathcal{K}}_{\mathcal{X}})\cdot\overline{\mathcal{L}}^{n}=-\frac{1}{(n+1)!}\overline{\mathcal{L}}^{n+1}+\frac{1}{2}\int\log(\frac{MA(\psi)}{e^{-\psi}})(dd^{c}\psi)^{n}/n!,\label{eq:decomp of arithm Mab in Fano case}
\end{equation}
where, in the last equality, we have exploited that $\mathcal{L}+\mathcal{K}_{\mathcal{X}}$
is trivial so that formula \ref{eq:arithm inters form for trivial}
applies. Applying formula \ref{eq:toric arithm vol as beautiful E}
to the first term in the rhs above thus gives 
\[
2\mathcal{M}_{(\mathcal{X},\mathcal{L})}(\psi):=-\mathcal{E}_{\psi_{P}}(\psi)+\int\log(\frac{MA(\psi)}{e^{-\psi}})(dd^{c}\psi)^{n}/n!=
\]
\[
=\mathrm{vol}(-K_{X})\left(-\frac{1}{V(X)}\mathcal{E}_{\psi_{P}}(\psi)+\left\langle \psi-\psi_{P},MA(\psi)\right\rangle +\int\log(\frac{MA(\psi)}{e^{-\psi_{P}}})MA(\psi)\right).
\]
 The rhs in the last equation above equals $\mathcal{M}_{\psi_{P}}(\psi)$
(by definition \ref{eq:def of M psi not}). Invoking Lemma \ref{lem:M psi P is Donaldson plus log}
thus concludes the proof.
\end{proof}
Next, consider an arithmetic Fano variety $\mathcal{X}$ (defined
in Section \ref{subsec:Arithmetic-Fano-varieties}). Denote by $\mathcal{\hat{D}}_{\Z}(\psi)$
the functional defined by formula \ref{eq:def of Ding Z}, corresponding
to the integral model $\mathcal{L}=-\mathcal{K}_{\mathcal{X}}.$ In
this arithmetic setup the following variants of the inequality \ref{eq:F geq D}
and the identity \ref{eq:inf is inf given reference} hold. 
\begin{prop}
\label{prop:inf arithm Mab vs Ding}When $\mathcal{L}=-\mathcal{K}_{\mathcal{X}}$
the following relations hold: 
\[
2\mathcal{M}_{(\mathcal{X},-\mathcal{K}_{\mathcal{X}})}\geq\mathrm{vol}(-K_{X})\mathcal{\hat{D}}_{\Z}
\]
 and 
\[
\inf_{\mathcal{C}^{0}(L)\cap\text{PSH\ensuremath{(L)}}}2\mathcal{M}_{(\mathcal{X},\mathcal{L})}=\mathrm{vol}(-K_{X})\inf_{\mathcal{C}^{0}(L)\cap\text{PSH\ensuremath{(L)}}}\mathcal{\hat{D}}_{\Z}.
\]
\end{prop}

\begin{proof}
First note that the second term in the decomposition \ref{eq:decomp of arithm Mab in Fano case}
of $\mathcal{M}_{(\mathcal{X},-\mathcal{K}_{\mathcal{X}})}(\psi)$
is precisely the entropy of $(dd^{c}\psi)^{n}/n!$ relative to $e^{-\psi}:$
\[
\mathcal{M}_{(\mathcal{X},-\mathcal{K}_{\mathcal{X}})}(\psi)=-\frac{(\mathcal{L},\psi)^{n+1}}{(n+1)!}+\text{Ent}_{e^{-\psi}}\left((dd^{c}\psi)^{n}/n!\right).
\]
 Since the entropy between two probability measure is non-negative
(by Jensen's inequality) this proves the inequality in the proposition
when the measure $e^{-\psi}$ has unit total volume. The general case
then follows from a simple scaling argument. Next, to prove the identity
in the proposition fix a reference metric $\psi_{0}$ in $\mathcal{H}(X,-K_{X})$
and rewrite the previous formula as
\begin{equation}
\frac{\mathcal{M}_{(\mathcal{X},-\mathcal{K}_{\mathcal{X}})}(\psi)}{\mathrm{vol}(-K_{X})}=-\left(\frac{(\mathcal{L},\psi)^{n+1}}{(n+1)!\mathrm{vol}(-K_{X})}+\left\langle \psi-\psi_{0},MA(\psi)\right\rangle \right)+\frac{1}{2}\text{Ent}_{e^{-\psi_{0}}}\left(MA(\psi)\right).\label{eq:arithm Mab as free energy}
\end{equation}
Accordingly, expressing $(\mathcal{L},\psi)^{n+1}=(\mathcal{L},\psi_{0})^{n+1}+(n+1)!\mathcal{E}_{\psi_{0}}(\psi)/2,$
using Lemma \ref{lem:change of metric}, gives
\[
\frac{\mathcal{M}_{(\mathcal{X},-\mathcal{K}_{\mathcal{X}})}(\psi)}{\mathrm{vol}(-K_{X})}=-\frac{1}{2}F_{\psi_{0}}\left(MA(\psi)\right)-\frac{1}{(n+1)!}(\mathcal{L},\psi_{0})^{n+1},
\]
 where $F_{\psi_{0}}(\mu)$ is the free energy functional \ref{eq:def of fri F}.
The proof is thus concluded by invoking the identity \ref{eq:inf is inf given reference}
and using Lemma \ref{lem:change of metric} again.
\end{proof}
\begin{rem}
When $-K_{X}$ admits a Kähler-Einstein metric $\phi_{KE}$ both infima
in the previous proposition are attained at $\phi_{KE}$ \cite{bbegz}.
The identity then follows directly from the Kähler-Einstein equation,
giving $MA(\phi_{KE})=e^{-\phi_{KE}},$ when $\phi_{KE}$ is volume-normalized. 
\end{rem}

In Section \ref{subsec:Outlook-on-a} the inequality in the previous
proposition will be generalized to any model $(\mathcal{X},\mathcal{L})$
of $(X,-K_{X}),$ by introducing an arithmetic Ding functional $\mathcal{D}_{(\mathcal{X},\mathcal{L})}$,
coinciding (up to normalization) with the functional $\mathcal{\hat{D}}_{\Z}$
under the conditions in the previous proposition.

\section{\label{sec:Comparison-with-the}Discussion and outlook}

\subsection{The function field analog}

Recall that, according to the philosophy of Arakelov geometry, the
function field analog of a metrized arithmetic variety $\mathcal{X}\rightarrow\text{Spec \ensuremath{\Z}}$
is a flat projective morphism
\[
\mathcal{\mathscr{X}}\rightarrow\mathcal{\mathscr{B}}
\]
from a normal complex projective variety $\mathcal{\mathscr{X}}$
to a fixed regular complex projective curve $\mathcal{\mathscr{B}}.$
In particular, the analog of the setup of arithmetic Fano varieties
in Conjecture \ref{conj:height intro} appears when $\mathcal{\mathscr{X}}$
is normal, the relative anti-canonical divisor $-\mathcal{\mathscr{K}}_{\mathscr{X}/\mathscr{B}}$
defines a relatively ample $\Q-$line bundle and the generic fiber
is K-semistable. The analog of the inequality in Conjecture \ref{conj:height intro}
does hold in this situation, but not the uniqueness statement. More
precisely, if $(X,-K_{X})$ is assumed K-semistable then it follows
from \cite{c-p} (see the beginning of \cite[Section 1.7.1]{c-p})
that
\begin{equation}
(-\mathcal{\mathscr{K}}_{\mathscr{X}/\mathscr{B}})^{n+1}\leq0.\label{eq:function field ineq}
\end{equation}
 Equality holds for the trivial fibrations $\mathcal{\mathscr{X}}=X\times\mathcal{\mathscr{B}}$
for any K-semistable $X.$ In particular, 
\begin{equation}
(-\mathcal{\mathscr{K}}_{\mathscr{X}/\mathscr{B}})^{n+1}\leq(-\mathcal{\mathscr{K}}_{\P^{n}\times\mathcal{\mathscr{B}}/\mathscr{B}})^{n+1}(=0)\label{eq:funct field anal of conj}
\end{equation}
 which is the function field analog of the inequality in Conjecture
\ref{conj:height intro}. Note that when $\mathcal{\mathscr{B}}=\P^{1}$
and the standard $\C^{*}-$action on $\P^{1}$ lifts to $\mathcal{\mathscr{X}},$
the inequality \ref{eq:function field ineq} follows directly from
the definition of K-semistability.
\begin{rem}
The analog of the volume-normalization (appearing in Conjecture \ref{conj:height intro})
is automatically satisfied in the function field case. Indeed, the
second term in the corresponding Ding functional $\mathcal{D}_{(\mathcal{X}_{\mathscr{X}/\mathscr{B}},-\mathcal{\mathscr{K}}_{\mathscr{X}/\mathscr{B}})},$
discussed in the following section, then vanishes.
\end{rem}

In contrast to Conjecture \ref{conj:height intro} projective space
thus plays no special role in the function field case (since equality
holds in the inequality \ref{eq:funct field anal of conj} for \emph{any}
product $\mathcal{\mathscr{X}}=X\times\mathcal{\mathscr{B}}).$ Conversely,
it should be stressed that the analog of the inequality \ref{eq:function field ineq}
\emph{fails} in the arithmetic situation (by the strict positivity
in Lemma \ref{lem:arithm vol of proj etc}). Hence, the function field
analogy is somewhat deceptive. Our general motivation for Conjecture
\ref{conj:height intro} is rather the analogy with the corresponding
result over $\C$ (corresponding to the trivial morphism $X\rightarrow\text{Spec \ensuremath{\C})}$
and the fact that projective space maximizes the degree of $-K_{X}$
\cite{fu}, among K-semistable $X$ of a given dimension (as well
as a range of other positivity properties of $-K_{X}$; see, for example,
the discussion and references in the introduction of \cite{l-z}). 

\subsection{\label{subsec:Outlook-on-a}A generalization of Conjecture \ref{conj:height intro}}

Consider a Fano variety $X_{F}$ defined over a number field $F,$
i.e. a field extension $F$ of $\Q$ of finite degree $[F:\Q].$ Let
$(\mathcal{X},\mathcal{L})$ be a normal polarized model of $(X_{F},-K_{X_{F}})$
over the ring of integers $\mathcal{O}_{F}$ of $F$ such that $\mathcal{K}_{\mathcal{X}/\text{Spec\ensuremath{\mathcal{O}_{F}}}}$
is defined as a $\Q-$line bundle. We will denote by $\psi$ a collection
of continuous psh $\psi_{\sigma}$ metrics on $-K_{X_{\sigma}},$
where $\sigma$ ranges over all embeddings of the field $F$ into
$\C$ and $X_{\sigma}$ denotes the corresponding complex projective
varieties. To the model $(\mathcal{X},\mathcal{L})$ we attach an
\emph{arithmetic Ding functional,} defined as follows. First consider
a model $(\mathcal{X},\mathcal{L})$ of $(X_{F},-K_{X_{F}})$ such
that $\mathcal{L}+\mathcal{K}_{\mathcal{X}/\text{Spec\ensuremath{\mathcal{O}_{F}}}}$
defines a bona fide line bundle. Then 
\[
\mathcal{D}_{(\mathcal{X},\mathcal{L})}:=\frac{[F:\Q](-K_{X})^{n}}{n!}\mathcal{\hat{D}}_{(\mathcal{X},\mathcal{L})}(\psi),
\]
 where $\mathcal{\hat{D}}_{(\mathcal{X},\mathcal{L})}(\psi)$ is the
\emph{normalized arithmetic Ding functional} defined by
\[
\mathcal{\hat{D}}_{(\mathcal{X},\mathcal{L})}(\psi):=-\frac{(\mathcal{L},\psi)^{n+1}}{[F:\Q](n+1)(-K_{X})^{n}}+\frac{1}{[F:\Q]}\widehat{\deg}\pi_{*}(\mathcal{L}+\mathcal{K}_{\mathcal{X}/\text{Spec\ensuremath{\mathcal{O}_{F}}}}),
\]
where the second term above denotes the arithmetic (Arakelov) degree
of the line bundle $\pi_{*}(\mathcal{L}+\mathcal{K}_{\mathcal{X}/\text{Spec\ensuremath{\mathcal{O}_{F}}}})\rightarrow\text{Spec}\mathcal{O}_{F},$
endowed with the $L^{2}-$metric induced by the metric $\psi$ on
$\mathcal{L}$ (i.e. on $-K_{X}).$ More generally, when $\mathcal{K}_{\mathcal{X}/\text{Spec\ensuremath{\mathcal{O}_{F}}}}$
is merely defined as a $\Q-$line bundle we fix a positive integer
$r$ such that $r(\mathcal{L}+\mathcal{K}_{\mathcal{X}/\text{Spec\ensuremath{\mathcal{O}_{F}}}})$
is defined as a line bundle and replace $\widehat{\deg}\pi_{*}(\mathcal{X},(\mathcal{L}+\mathcal{K}_{\mathcal{X}/\text{Spec\ensuremath{\mathcal{O}_{F}}}})$
with $r^{-1}\widehat{\deg}\pi_{*}(\mathcal{X},(r(\mathcal{L}+\mathcal{K}_{\mathcal{X}/\text{Spec\ensuremath{\mathcal{O}_{F}}}})),$
where now $\pi_{*}(r(\mathcal{L}+\mathcal{K}_{\mathcal{X}/\text{Spec\ensuremath{\mathcal{O}_{F}}}}))$
is endowed with the $L^{2/r}-$metric induced by $\psi.$ Concretely,
given a rational global section $s_{r}$ of $\pi_{*}(r(\mathcal{L}+\mathcal{K}_{\mathcal{X}/\text{Spec\ensuremath{\mathcal{O}_{F}}}})),$
one may express
\begin{equation}
\widehat{\deg}\pi_{*}(r(\mathcal{L}+\mathcal{K}_{\mathcal{X}/\text{Spec\ensuremath{\mathcal{O}_{F}}}}))=-\frac{1}{2}\sum_{\sigma}\log\int_{X_{\sigma}}|s_{r}|^{2/r}e^{-\psi_{\sigma}}+\sum_{\mathfrak{p}}\text{ord}_{\mathfrak{p}}(s_{r})\log|\mathfrak{p}|,\label{eq:formula for arithm degree of direct}
\end{equation}
where $|s_{r}|^{2/r}e^{-\psi_{\sigma}}$ denotes corresponding measure
on $X_{\sigma},$ $\text{ord}_{\mathfrak{p}}(s)$ denotes the order
of vanishing of $s_{r}$ at the closed point $\mathfrak{p}$ in $\text{Spec}\mathcal{O}_{F}$
and $|\mathfrak{p}|$ donotes the norm of the prime ideal in\emph{
$\mathcal{O}_{F}$ }defined by $\mathfrak{p}$ (i.e., the cardinality
of the corresponding residue field $\mathcal{O}_{F}/\mathfrak{p}$
). The functional $\mathcal{\hat{D}}_{(\mathcal{X},\mathcal{L})}$
thus coincides with the functional $\mathcal{\hat{D}}_{\Z},$ defined
in formula \ref{eq:def of Ding Z}, up to an additive constant and
a factor of two. Note that when $F=\Q$ and $\mathcal{L}=-\mathcal{K}_{\mathcal{X}}$
we have $2\widehat{\deg}\pi_{*}(\mathcal{L}+\mathcal{K}_{\mathcal{X}/\text{Spec\ensuremath{\mathcal{O}_{F}}}})=-\log\int_{X}e^{-\phi}.$
Indeed, in this we can take $s_{r}=1\in H^{0}(\mathcal{X},\mathcal{O}_{\mathcal{X}}),$
which is globally non-vanishing, by Lemma \ref{lem:Stein}.
\begin{rem}
The functional $\mathcal{D}_{(\mathcal{X},\mathcal{L})}(\psi)$ is
the arithmetic analog of the degree of the Ding line bundle of a test
configuration $(\mathscr{X},\mathscr{L})$ for $(X,-K_{X})$ introduced
in \cite{ber0}. As shown in \cite{fu2} a Fano variety $X$ is K-semistable
iff the degree of the Ding line bundle is non-negative for any test
configuration $(\mathscr{X},\mathscr{L}).$
\end{rem}

Now consider the following invariant of the Fano variety $X_{F}:$
\[
\mathcal{D}(X_{F}):=\inf\left([F:\Q]^{-1}\mathcal{D}_{(\mathcal{X},\mathcal{L})}\right),
\]
 where the inf runs over all integral models $(\mathcal{X},\mathcal{L})$
of $(X,-K_{X})$ and metrics $\psi$ as above. We propose the following
generalization of Conjecture \ref{conj:height intro}:
\begin{conjecture}
\label{conj:arithm Ding}Let $X_{F}$ be a K-semistable Fano variety
defined over a number field $F.$ Then 
\[
\mathcal{D}(X_{F})\geq\mathcal{D}_{(\P_{\Z}^{n},-\mathcal{K}_{\P_{\Z}^{n}})}(\psi_{FS}),
\]
 where $\psi_{FS}$ denotes the volume-normalized Fubini-Study metric
$\psi_{FS}$ on $-K_{\P^{n}}.$ Equivalently, for any model $(\mathcal{X},\mathcal{L})$
and continuous psh metric $\psi$, normalized so that $\widehat{\deg}\pi_{*}(\mathcal{L}+\mathcal{K}_{\mathcal{X}/\text{Spec\ensuremath{\mathcal{O}_{F}}}})=0,$
\[
\frac{1}{[F:\Q]}(\mathcal{L},\psi)^{n+1}\leq(-\mathcal{K}_{\P_{\Z}^{n}},\psi_{FS})^{n+1}.
\]
Moreover, equality holds if and only if $(\mathcal{X},\mathcal{L})$
is isomorphic to $(\P_{\mathcal{O}_{F}}^{n},-\mathcal{K}_{\P_{\mathcal{O}_{F}}^{n}}+\pi^{*}M)$
for some line bundle $M\rightarrow\text{Spec \ensuremath{\mathcal{O}_{F}}}$
and $\psi$ coincides with $\psi_{FS},$ up to the action of an automorphism
of $\P^{n}$. 
\end{conjecture}

Note that, in general, $\mathcal{D}_{(\mathcal{X},\mathcal{L})}(\psi)=\mathcal{D}_{(\mathcal{X},\mathcal{L}+\pi^{*}M)}(\psi)$
for any line bundle $M\rightarrow\text{Spec \ensuremath{\mathcal{O}_{F}.} }$
We expect - inspired by Odaka's conjecture discussed in Section \ref{subsec:The-arithmetic-K-energy}
- that any integral model $(\mathcal{X},\mathcal{L})$ which is globally
K-semistable realizes the infimum defining the invariant $\mathcal{D}(X_{F}).$

Next, given a polarized scheme $(\mathcal{X},\mathcal{L})$ over a
number field $F,$ we will, as in the case $F=\Q,$ denote by $\mathcal{M}_{(\mathcal{X},\mathcal{L})}(\psi)$
the arithmetic Mabuchi funtional defined by the intersection-theoretic
expression in formula \ref{eq:def of arithm Mab}. In general, the
following inequality between the arithmetic Mabuchi functional and
the arithmetic Ding functional holds, showing, in particular, that
Conjecture \ref{conj:arithm Ding} implies Conjecture \ref{conj:min of Odaka for Fano}
concerning Odaka's modular invariant. The inequality can be viewed
as an arithmetic analogy of the inequality for test configurations
in \cite[Lemma 3.10]{ber0}.
\begin{prop}
\label{prop:Mab greater than Ding}If $(\mathcal{X},\mathcal{L})$
is a normal polarized model of $(X,-K_{X})$ over $\text{Spec \ensuremath{\mathcal{O}_{F}} }$
which is $\Q-$Gorenstein, then
\[
\mathcal{M}_{(\mathcal{X},\mathcal{L})}(\psi)\geq\mathcal{D}_{(\mathcal{X},\mathcal{L})}(\psi)
\]
 with equality iff $\psi$ is a Kähler-Einstein metric and $\mathcal{L}$
is isomorphic to $-\mathcal{K}_{\mathcal{X}/\text{Spec\ensuremath{\mathcal{O}_{F}}}}\otimes\pi^{*}M$
for some line bundle $M$ over $\text{Spec}\ensuremath{\mathcal{O}_{F}.}$
\end{prop}

\begin{proof}
To simplify the notation we assume that $r=1$ (but the proof in the
general case is essentially the same). It follows directly from the
definitions that we need to prove that
\begin{equation}
\frac{1}{L^{n}}(\overline{\mathcal{L}}+\overline{\mathcal{K}})\cdot\overline{\mathcal{L}}^{n}-\widehat{\deg}\pi_{*}(\mathcal{L}+\mathcal{K}_{\mathcal{X}/\text{Spec\ensuremath{\mathcal{O}_{F}}}})\geq0\label{eq:pf Mab greather than Ding}
\end{equation}
 with equality iff the conditions in the proposition hold. First observe
that the left hand side above is invariant when $\mathcal{L}$ is
replaced by $\mathcal{L}+\pi^{*}M,$ where $M$ is any line bundle
over $\text{Spec \ensuremath{\mathcal{O}_{F}.} }$Hence, we may as
well assume that $\pi_{*}(\mathcal{L}+\mathcal{K}_{\mathcal{X}/\text{Spec\ensuremath{\mathcal{O}_{F}}}})$
admits a global regular section $s$ that is non-vanishing over the
generic fiber. Now, by the restriction formula for arithmetic intersection
numbers \cite[Prop 2.3.1]{b-g-s},
\begin{equation}
\frac{1}{L^{n}}(\overline{\mathcal{L}}+\overline{\mathcal{K}})\cdot\overline{\mathcal{L}}^{n}=\frac{1}{2}\int_{X(\C)}\log(\frac{MA(\psi)}{|s|^{2}e^{-\psi}})MA(\psi)+\frac{1}{L^{n}}(s=0)\cdot\overline{\mathcal{L}}^{n},\label{eq:restr formula}
\end{equation}
where $(s=0)$ denotes the subscheme of $\mathcal{X}$ cut out by
$s.$ By Jensen's inequality,
\begin{equation}
\int_{X(\C)}\log(\frac{MA(\psi)}{|s|^{2}e^{-\psi}})MA(\psi)\geq-\frac{1}{2}\sum_{\sigma}\log\int_{X_{\sigma}}|s|^{2}e^{-\psi}=\widehat{\deg}\pi_{*}(\mathcal{L}+\mathcal{K}_{\mathcal{X}/\text{Spec\ensuremath{\mathcal{O}_{F}}}})-\sum_{\mathfrak{p}}\text{ord}_{\mathfrak{p}}(s)\log|\mathfrak{p}|.\label{eq:entropy term}
\end{equation}
 Hence, decomposing the subscheme $(s=0)$ of $\mathcal{X}$ as a
sum of effective divisors $E_{\mathfrak{p}},$ where $E_{\mathfrak{p}}$
is supported on the fiber $\mathcal{X}_{\mathfrak{p}}$ of $\mathcal{X}$
over $\mathfrak{p},$ 
\[
\frac{1}{L^{n}}(\overline{\mathcal{L}}+\overline{\mathcal{K}})\cdot\overline{\mathcal{L}}^{n}\geq\frac{1}{L^{n}}(s=0)\cdot\overline{\mathcal{L}}^{n}-\sum_{\mathfrak{p}}\text{ord}_{\mathfrak{p}}(s)\log|\mathfrak{p}|=\left(\frac{1}{L^{n}}\mathcal{L}_{|\mathcal{X}_{\mathfrak{p}}}^{n}\cdot E_{\mathfrak{p}}-\sum_{\mathfrak{p}}\text{ord}_{\mathfrak{p}}(s)\right)\log|\mathfrak{p}|,
\]
 using, again, the restriction formula in the last equality. Since
$\text{ord}_{\mathfrak{p}}(s)\geq0,$ we can express $E_{\mathfrak{p}}=E'_{\mathfrak{p}}+\text{ord}_{\mathfrak{p}}(s)\mathcal{X}_{\mathfrak{p}}$
for an effective divisor $E'_{\mathfrak{p}},$ giving
\[
\frac{1}{L^{n}}(\overline{\mathcal{L}}+\overline{\mathcal{K}})\cdot\overline{\mathcal{L}}^{n}\geq\left(\frac{1}{L^{n}}\mathcal{L}_{|\mathcal{X}_{\mathfrak{p}}}^{n}\cdot E'_{\mathfrak{p}}\right)\log|\mathfrak{p}|\geq0.
\]
 Finally, equality holds in the inequality \ref{eq:entropy term}
iff $MA(\psi)$ is proportional to $|s|^{2}e^{-\psi}$ for all $X_{\sigma},$
i.e. iff $\psi$ is Kähler-Einstein. Moreover, since $\mathcal{L}$
is relatively ample the right hand side in the last inequality above
vanishes iff $E'_{\mathfrak{p}}$ is the zero-divisor for all $\mathfrak{p}$
i.e. iff ($s=0)$ is a linear combination of fibers $\mathcal{X}_{\mathfrak{p}}$
and thus $\mathcal{L}+\mathcal{K}$ is isomorphic to $\pi^{*}M$ for
some line bundle $M$ over $\text{Spec}\ensuremath{\mathcal{O}_{F}.}$
\end{proof}

\subsection{Comparison with bounds on Bost-Zhang's normalized heights}

The normalized arithmetic Ding functional $\mathcal{\hat{D}}_{(\mathcal{X},\mathcal{L})}$
is reminiscent of Bost's normalized height $h_{\text{norm }},$ introduced
in \cite{bo2} in the general setup of polarized variety $(X_{F},L_{F})$
defined over a number field $F:$ 
\[
h_{\text{norm }}(\mathcal{L},\psi):=\frac{(\mathcal{L},\psi)^{n+1}}{[F:\Q](n+1)(L_{F})^{n}}-\frac{1}{[F:\Q]N}\widehat{\deg}\pi_{*}\mathcal{X},
\]
 assuming that the rank $N$ of the vector bundle $\pi_{*}\mathcal{L}\rightarrow\text{Spec}\mathcal{O}_{F}$
is non-zero and $\pi_{*}(\mathcal{X},\mathcal{L})$ is endowed with
the $L^{2}-$norm induced by the continuous psh metrics $\psi_{\sigma}$
on $L_{\sigma}$ and the volume forms $MA(\psi_{\sigma})$ on $X_{\sigma}$
(defined by formula \ref{eq:def of MA}). When $L_{F}$ is very ample
it is shown in \cite{bo2} that the functional $h_{\text{norm }}(\mathcal{L},\cdot)$
is bounded from below iff the Chow point of $(X_{F},L_{F})$ is semistable
wrt the action of the group $GL(N,F)$ on the Chow variety (in the
sense of Geometric Invariant Theory). More precisely, it it shown
in \cite{bo2} that the semi-stability in question is equivalent to
a lower bound on Bost's intrinsic normalized height of $(X_{F},L_{F}):$
\[
\inf h_{\text{norm }}>-\infty
\]
where the infimum runs over all models $(\mathcal{X},\mathcal{L})$
and metrics $\psi$ as above. In fact, by \cite[Prop 2.1]{bo2} and
\cite[Thm 4.4]{zh1} the Chow-semistability in question is equivalent
to the following explicit lower bound: 
\begin{equation}
h_{\text{norm }}(\mathcal{L},\psi)\geq-\frac{1}{2}\sum_{n=1}^{N+1}\sum_{m=1}^{n}\frac{1}{m}-\frac{1}{2}\log N\label{eq:Zhangs lower bound}
\end{equation}
(it is moreover conjectured in \cite{zh1} that the first term in
the right hand side above may be replaced by $0$$).$ 

In this setup the role of the normalization $\widehat{\deg}\pi_{*}(\mathcal{L}+\mathcal{K}_{\mathcal{X}/\text{Spec\ensuremath{\mathcal{O}_{F}}}})=0$
in Conjecture \ref{conj:arithm Ding} is thus played by the normalization
$\widehat{\deg}\pi_{*}\mathcal{L}=0.$ However, in contrast to Conjecture
\ref{conj:arithm Ding} the lower bound \ref{eq:Zhangs lower bound}
on $h_{\text{norm }}(\mathcal{L},\psi)$ corresponds to a \emph{lower}
bound on $(\mathcal{L},\psi)^{n+1}$ for any normalized metric. Note
also that one virtue of the normalization condition in Conjecture
\ref{conj:arithm Ding} is that it is comparatively explicit, since
$\pi_{*}(\mathcal{L}+\mathcal{K}_{\mathcal{X}/\text{Spec\ensuremath{\mathcal{O}_{F}}}})$
has rank one (so that formula \ref{eq:formula for arithm degree of direct}
applies, showing that it is enough to assume that the volume forms
$|s_{r}|^{2/r}e^{-\psi_{\sigma}}$ on $X_{\sigma}$ are normalized).
Another advantage of this normalization condition is that it applies
to any continuous metric $\psi$ (at the price of replacing $(\mathcal{L},\psi)^{n+1}$
with the $\chi-$arithmetic volume of $\mathcal{L},$ as in Theorem
\ref{thm:arithm Vol and K semi st}).

Finally, we recall that when $\mathcal{L}$ is replaced by $k\mathcal{L}$
for a large positive integer $k$ it follows from \cite[Thm 3.7]{o}
that there exists constants $a$ and $b$ (depending only on $(X_{F},L_{F})$)
such that $a>0$ 
\begin{equation}
\mathcal{M}_{(\mathcal{X},\mathcal{L})}(\psi)/L^{n}=h_{\text{norm }}(k\mathcal{L},\psi)-a\log N_{k}+b+o(1),\label{eq:Od asym}
\end{equation}
as $k\rightarrow\infty,$ where $N_{k}$ denotes the rank of $H^{0}(\mathcal{X},k\mathcal{L})$
which diverges as $k\rightarrow\infty.$ Unfortunately, the diverging
term $a\log N_{k}$ makes it impossible to infer lower bounds on $\mathcal{M}_{(\mathcal{X},\mathcal{L})}(\psi)$
from lower bounds on $h_{\text{norm }}(k\mathcal{L}).$ Since $\mathcal{M}_{(\mathcal{X},\mathcal{L})}(\psi)$
coincides with $\mathcal{D}_{(\mathcal{X},\mathcal{L})}(\psi)$ when
$\mathcal{L}$ equals $-\mathcal{K}_{\mathcal{X}/\text{Spec\ensuremath{\mathcal{O}_{F}}}}$this
means that Conjecture \ref{conj:arithm Ding} can not be deduced from
bounds of the form \ref{eq:Zhangs lower bound} by letting $k$ (and
hence $N)$ tend to infinity. 

\subsection{\label{subsec:Comparison-with-Odaka's Falting}Comparison with Odaka's
and Faltings' modular heights }

Finally, let us compare our normalizations of the arithmetic Mabuchi
functional with those of Odaka \cite{od2} and Faltings \cite{fa2}.
First of all our multiplicative normalization for the arithmetic Mabuchi
functional $\mathcal{M}_{(\mathcal{X},\mathcal{L})}$ (formula \ref{eq:def of arithm Mab intro})
are made so that $\pm\mathcal{M}_{(\mathcal{X},\pm K_{\mathcal{X}})}=(\pm\mathcal{K}_{\mathcal{X}})^{n+1}/(n+1).$
Moreover, as discussed in Section \ref{subsec:Odaka's-modular-height},
we are employing the metric on $-K_{X}$ induced by the \emph{normalized}
volume form $\omega^{n}/L^{n}$ of the Kähler form $\omega$ defined
by a given metric $\psi$ on $\mathcal{L}$ with positive curvature
(i.e. $\omega=dd^{c}\psi).$ Comparing with Odaka's arithmetic Mabuchi
functional, that we shall denote by $\mathcal{M}_{(\mathcal{X},\mathcal{L})}^{(O)}(\psi),$
thus yields

\begin{equation}
\frac{1}{(n+1)!L^{n}}\mathcal{M}_{(\mathcal{X},\mathcal{L})}^{(O)}=\mathcal{M}_{(\mathcal{X},\mathcal{L})}+\frac{1}{2}\frac{L^{n}}{n!}\log(L^{n}/n!).\label{eq:Mab O in terms of Mab}
\end{equation}
 In the case that $\mathcal{X}$ is an abelian variety it was shown
in \cite{od2} that the infimum of Odaka's arithmetic Mabuchi functional
over all metrics on $\mathcal{L}$ with positive curvature coincides
with Faltings' (modular) height \cite{fa2}, up to a multiplicative
and an additive constant depending on $L^{n}.$ Here we note that
our normalizations are consistent with those of Faltings: 
\begin{prop}
Let $\mathcal{X}$ be a projective and flat scheme over $\Z$ and
assume that $\mathcal{K_{\mathcal{X}}}$ is trivial. For any relatively
ample line bundle $\mathcal{L}$ over \emph{$\mathcal{X}$}
\begin{equation}
\inf_{\psi}\frac{1}{L^{n}/n!}\mathcal{M}_{(\mathcal{X},\mathcal{L})}(\psi)=-\frac{1}{2[\F:\Q]}\log\frac{1}{2^{n}}\left|\int_{X(\C)}\Omega\wedge\bar{\Omega}\right|,\label{eq:inf Mab is Falt}
\end{equation}
 where $\Omega$ is a generator of $H^{0}(\mathcal{X},\mathcal{K}_{\mathcal{X}})$
and the inf ranges over all psh metrics $\psi$ on $\mathcal{L}$
and $V:=L^{n}/n!.$ 
\end{prop}

\begin{proof}
This is essentially equivalent to \cite[Thm 2.11]{od2}, using the
relation \ref{eq:Mab O in terms of Mab}. Anyhow, in order to verify
that all normalizations are consistent we provide a simple direct
proof. Assume, to simplify the notation, that $\F=\Q.$ Recall that
Faltings' modular height \cite{fa2} is defined as the arithmetic
degree of $\pi_{*}(\mathcal{X},K_{\mathcal{X}}),$ with respect to
the $L^{2}-$metric on $H^{0}(X,K_{X})$ defined by $\left\Vert \Omega\right\Vert ^{2}:=\frac{1}{2^{n}}\left|\int_{X(\C)}\Omega\wedge\bar{\Omega}\right|.$
This is precisely the right hand side in formula \ref{eq:inf Mab is Falt}.
As for the left hans side it is is given by 
\[
\int_{X}\log\left(\frac{(dd^{c}\psi)^{n}/Vn!}{\frac{i^{n^{2}/2}}{2^{n}}\Omega\wedge\bar{\Omega}/\left\Vert \Omega\right\Vert ^{2}}\right)\frac{(dd^{c}\psi)^{n}}{Vn!}=\int_{X}\log\left(\frac{(dd^{c}\psi)^{n}/Vn!}{\frac{i^{n^{2}/2}}{2^{n}}\Omega\wedge\bar{\Omega}/\left\Vert \Omega\right\Vert ^{2}}\right)\frac{(dd^{c}\psi)^{n}}{Vn!}-\log\left\Vert \Omega\right\Vert ^{2}.
\]
(as follows readily from the definitions, just as in formula \ref{eq:restr formula}).
Now, by Jensen's inequality this expression is minimal precisely when
the two probability measures $(dd^{c}\psi)^{n}/Vn!$ and $2^{-n}i^{n^{2}/2}\Omega\wedge\bar{\Omega}/\left\Vert \Omega\right\Vert ^{2}$
coincide, which, equivalently, means that $dd^{c}\psi$ is a Kähler-Einstein
metric. By the Calabi-Yau theorem such a metric exists for any given
ample $L,$ which concludes the proof.
\end{proof}
The previous proposition has the following consequence, when combined
with well-known properties of Faltings' modular height of abelian
varieties (cf. the discussion in relation to \cite[Thm 2.11]{od2}
and \cite[Section 2.3.2]{od2}). Consider a polarized abelian variety
$(X_{\F_{0}},L_{\F_{0}})$ defined over a given number field $\F_{0}.$
Then the infimum of $\text{vol}(L)^{-1}\mathcal{M}_{(\mathcal{X},\mathcal{L})}$
over all metrics, finite field extensions $\F,$ models over $\mathcal{O}_{\F}$
and positively curved metrics on $L\rightarrow X_{\F}(\C)$ is attained
at any semi-stable reduction of the Néron\emph{ }model $\mathcal{X}$
of $X_{\F},$ when $L$ is endowed with a Kähler-Einstein metric.
Moreover, in the particular case of elliptic curves it was observed
in \cite[page 29]{del} that the minimal value of the aforementioned
infimum over all $X_{\F}$ is attained at the semistable reduction
of the Néron model $\mathcal{X}_{0}$ of any elliptic curve with vanishing
$j-$invariant ($\mathcal{X}_{0}$ is uniquely determined for any
sufficently large field extension). Thus the role of $\mathcal{X}_{0}$
among all models of elliptic curves, is somewhat analogous to the
role of $\P_{\Z}^{n}$ in Conjectures \ref{conj:height intro}, \ref{conj:min of Odaka for Fano}.
However, it should be stressed that in the setup of Fano varieties
the choice of multiplicative normalization is crucial. Indeed, while
$\P_{\Z}^{n}$ minimizes $\mathcal{M}_{(\mathcal{X},-\mathcal{K}_{\mathcal{X}})}(\psi_{\text{KE}})$
over the canonical toric integral models of all K-semistable toric
Fano varieties $X$ (assuming that $n\leq6)$ it does\emph{ }not\emph{
}minimize $\text{vol}(-K_{X})^{-1}\mathcal{M}_{(\mathcal{X},-\mathcal{K}_{\mathcal{X}})}(\psi_{\text{KE}}).$
In fact, for all we know it could actually be the case that $\text{vol}(-K_{X})^{-1}\mathcal{M}_{(\mathcal{X},-\mathcal{K}_{\mathcal{X}})}(\psi_{\text{KE}})$
is\emph{ maximal} on $\P_{\Z}^{n}.$ For example, this turns out to
be the case in the more general setup of Fano orbifolds (not assumed
toric) when $X$ has relative dimension one (a proof will appear in
a separate publication).

\end{document}